\documentclass{amsart}
\copyrightinfo{2006}{American Mathematical Society}
\usepackage{mathrsfs,amscd}

\newtheorem{theorem}{Theorem}[section]
\newtheorem{lemma}[theorem]{Lemma}

\theoremstyle{definition}

\theoremstyle{remark}

\numberwithin{equation}{section}
\allowdisplaybreaks
\begin{document}

\title{KAM for Hamiltonian partial differential equations with weaker Spectral Asymptotics}
\author[Y. Li]{Yong Li}
\author[L. Xu]{Lu Xu}
\address{College of Mathematics, Jilin University,  Changchun 130012, People's Republic of China}
\email{liyong@jlu.edu.cn;luxujilin@yahoo.cn}
\thanks{This work supported by NSFC Grant No 10531050, National 973 Project of China No 2006CD805903,
SRFDP Grant No 20090061110001, and the 985 Project of Jilin
University.}

\date{June 2011}

\begin{abstract}
In this paper, we establish an abstract infinite dimensional KAM theorem dealing with normal frequencies in weaker
spectral asymptotics
\begin{eqnarray*}
\Omega_{i}(\xi)=i^d+o(i^{d})+o(i^{\delta}),
\end{eqnarray*}
where $d>0,~\delta<0$, which can be applied to a large class of Hamiltonian partial differential equations in high dimensions.
As a consequence, it is proved that there exist many invariant tori and thus quasi-periodic solutions for Schr\"{o}dinger equations, the Klein-Gordon equations with exponential nonlinearity and other equations of any spatial dimension.
\end{abstract}
\maketitle

\section{Introduction}
The KAM theory in infinite dimension requires the spectral asymptotics assumption that exhibits the asymptotic separation of the spectrum. However,
H. Weyl's theory shows that the spectrum of the eigenvalue problem
\begin{eqnarray*}
\triangle\phi=-\lambda\phi,~~on~\Omega
\end{eqnarray*}
and
\begin{eqnarray*}
\phi|_{\partial\Omega}\equiv0
\end{eqnarray*}
has the following asymptotic behavior as $k\longrightarrow\infty$,
\begin{eqnarray}\label{E1-1}
\lambda_k\sim C_m(\frac{k}{V})^{\frac{2}{m}},
\end{eqnarray}
where $\Omega$ is a bounded domain in $\textbf{R}^m$,
$V$ denotes the volume of $\Omega$, $C_m=(2\pi)^2B_m^{-\frac{2}{m}}$ is the Weyl constant with $B_m=$ volume of the unit $m$-ball.

When the dimension $m>2$, it is easy to see that $0<\frac{2}{m}<1$ from (\ref{E1-1}). As a result, the spectral separation disappears.
Thus, a natural question is how to deal with this kind of
weaker spectral asymptotics when we consider the existence of lower dimensional KAM tori which is shown for a class of nearly
integrable Hamiltonian systems of infinite dimension. This is correspondent how to prove that there exist many invariant tori
and thus quasi-periodic solutions for a class of Hamiltonian partial differential equations in higher dimension with weaker
spectral asymptotics
\begin{eqnarray*}
\Omega_{i}(\xi)=i^d+o(i^{d})+o(i^{\delta}),~~d>0.
\end{eqnarray*}

The KAM (Kolmogorov-Arnold-Moser) theory is a very powerful tool to find periodic or quasi-periodic solution for higher dimensional Hamiltonian PDEs.
There have been many remarkable results. For reader's convenience, we refer to Kuksin \cite{kuksin,kuksin1},
Craig and Wayne(\cite{craig,wayne}),
Bourgain \cite{bourgain1,bourgain2,bourgain3}, and P\"oschel \cite{posel1,posel2}.
The first breakthrough result is made by Bourgain \cite{bourgain3} who proved that the two dimensional nonlinear Schr\"{o}dinger equations admit
small-amplitude quasi-periodic solutions. Later he improved in \cite{B} his method and proved that
the higher dimensional nonlinear Schr\"{o}dinger and wave equations admit small-amplitude quasi-periodic
solutions.

Constructing quasi-periodic solutions of higher dimensional Hamiltonian PDEs by method
from the finite dimensional KAM theory appeared later. The breakthrough of constructing quasi-periodic
solutions for more interesting higher dimensional Schr\"{o}dinger equation by modified
KAM method was made recently by Eliasson-Kuksin. They proved in \cite{eliasson} that the higher dimensional
nonlinear Schr\"{o}dinger equations admit small-amplitude linearly-stable quasi-periodic
solutions. Recently, Geng et al.\cite{geng1} extended the Bourgain's existence result \cite{bourgain3,B1} to arbitrary finite dimensional invariant tori.
They also got a nice linear normal form, which
can be used to study the linear stability of the obtained solutions.

However, the corresponding KAM type theorems on higher dimensional Hamiltonian PDEs that need
infinite spectrum
\begin{eqnarray*}
\Omega(\xi)=(\Omega_{1}(\xi),\cdots,\Omega_i(\xi),\cdots)_{i\in
Z_{+}^{1}}
\end{eqnarray*}
is generally assumed to be separated as
$i\rightarrow\infty$, namely, there exits $d\ge 1,~\delta<0$ such
that
\begin{eqnarray}\label{jianjinxing}
\Omega_i(\xi)=i^d+o(i^{d})+o(i^{\delta}),~~~\xi\in\Pi~({\rm
a~bounded~ region~in~}{\bf R}^n).
\end{eqnarray}
Usually, at each KAM step, one must remove the frequencies not
satisfying Diophantine conditions. The spectral asymptotics
guarantee that there exists a Cantor set $\Pi_{\gamma}\subset\Pi$,
for all $\xi\in\Pi_{\gamma}$, the tori with frequency $\omega(\xi)$
survive and the measure $|\Pi\setminus\Pi_{\gamma}|\rightarrow0$ as
$\gamma\rightarrow0$.

Now what happens to the persistence for weaker
spectral asymptotics, due to the spectral in higher dimensional PDEs
might destroy the separation by Weyl's asymptotics formula. The
motivation leads to whether one can  relax the spectral asymptotics
condition to the case $d>0$ or not. In the present paper, just at
this weaker spectral asymptotics, we prove the Melnikov persistence
for infinite-dimensional Hamiltonian systems.

In the usual KAM formulism to PDEs, one can make the perturbation
$P$ getting smaller by only truncating the order of the angle variable $x$.
However, to deal with the weaker spectral asymptotics, we not only
truncate the angle variable $x$ but also truncate the normal
variables $z,~\bar{z}$. This results in that we are forced to modify
the usual Diophantine condition. According to this new difficulty,
we introduce a KAM iteration that is  similar to Nash-Moser
iteration and somewhat different from previous ones.

We consider the Hamiltonian system
\begin{equation}\label{yiweihamidun}
H=\langle\omega(\xi),y\rangle+\sum_{j\geq1}\Omega_{j}(\xi)z_j\bar{z}_j
+P(x,y,z,\bar{z};\xi),
\end{equation}
where $(x,y,z,\bar{z})$ lies in the complex neighborhood
$$D_{a,p}(s,r)=\{(x,y,z,\bar{z}):|{\rm Im}~x|<r, |y|<s^2,
|z|^{a,p},|\bar{z}|^{a,p}<s\}$$of ${\bf T}^{n}\times \{0\}\times
\{0\}\times \{0\}$, ${\bf T}^{n}$ is an $n$-torus, and
$a\ge0,~p\ge0$, with the norm $|z|^{a,p}=\sqrt{\sum_{i\ge
1}|z_i|^{2}i^{2p}{\rm e}^{2ai}}$. The frequencies $\omega$ and
$\Omega$ depend on $\xi\in\Pi\subset {\bf R}^{n}$, $\Pi$ is a closed
bounded region.

With respect to the symplectic form
$$
\sum_{i=1}^{n}{\rm d}x_{i}\wedge {\rm d}y_{i}+\sum_{j=1}{\rm
d}z_{j}\wedge {\rm d}\bar{z}_{j},
$$
the associated unperturbed motion is simply described by
$$
\left\{\begin{array}{ll}
\dot x & =  \omega(\xi),\\
\dot y & =  0,\\
\dot z & =\Omega(\xi)\bar{z},\\
\dot{\bar{z}}&=-\Omega(\xi)z.
\end{array}\right.
$$

Assume that:
\begin{itemize}
\item[{\bf A1)}] {\it Nondegeneracy.}~~The map $\xi\mapsto\omega(\xi)$ is a
homemorphism and Lipschitz continuous in both directions. Moreover,
for all integer vectors $(k,l)\in {\bf Z}^{n}\times {\bf
Z}^{\infty}$ with $1\leq|l|\leq 2$,
$$
meas\{\xi:~\langle k,\omega(\xi)\rangle+\langle
l,\Omega(\xi)\rangle=0\}=0
$$
and
$$
\langle l,\Omega(\xi)\rangle\neq 0,~~\forall\xi\in\Pi,
$$
where $|l|=\sum_{j\geq 1}|l_{j}|$ for integer vectors.
\item[{\bf A2)}] {\it Spectral Asymptotics.}~~The components $\Omega_i\ne\Omega_j$ if $i\ne j$,
and there exist $d>0$ and $\delta<0$ such that
$$
\Omega_{j}(\xi)=j^{d}+o(j^d)+o(j^{\delta}),~~~j\rightarrow\infty,
$$
where the dots stand for lower order terms in $j$, allowing also
negative exponents.

Moreover, \begin{eqnarray*} |\omega|^{\mathcal
{L}}_{\Pi}+|\Omega|^{\mathcal {L}}_{-\delta,\Pi}\leq M<\infty,
\end{eqnarray*}
where
$$
|\omega|_{\pi}^{\mathcal
{L}}=\sup_{\xi\ne\varepsilon}\frac{|\omega(\xi)-\omega(\varepsilon)|}{|\xi-\varepsilon|},
$$
$$
|\Omega|^{\mathcal
{L}}_{-\delta,\Pi}=\sup_{\xi\ne\varepsilon}\sup_{i\ge
1}\frac{|\Omega_i(\xi)-\Omega_i(\varepsilon)|i^{-\delta}}{|\xi-\varepsilon|},
$$
for all $\xi\in\Pi$.

\item[{\bf A3)}]{\it Regularity.}~~The perturbation $P\in{\mathcal
{F}}_{a,p}^{\bar{a},\bar{p}}$, that is, $P$ is real analytic on the
space coordinate and Lipschitz on the parameters, and for each
$\xi$, Hamiltonian vector space field
$X_{P}=(P_{y},-P_{x},P_{\bar{z}},-P_{z})^{T}$ defines a real
analytic map
$$
X_{P}:~{\mathcal {P}}^{a,p}\rightarrow{\mathcal
{P}}^{\bar{a},\bar{p}},
$$
where $\bar{a}>a>0,~\bar{p}\ge p\ge0$, and $\mathcal{P}^{a,p}=\textbf{T}^n\times\textbf{R}^n\times\mathcal{L}^{a,p}\times\mathcal{L}^{a,p}$.
\end{itemize}

Suppose that $H=N+P$ satisfies ${\rm A1)-A3)}$. Then we have

\vspace{0.5cm} \noindent{\bf Theorem A.} {\it For given $r,s,a>0$,
if there exits a sufficiently small
$\mu=\mu(a,p,\bar{p},r,s,n,\tau)>0$ (or equivalently
$\mu_*=\mu_*(a,p,\bar{p},r,s,n,\tau)>0$) such that the perturbation
$P$ satisfies
\begin{eqnarray*}
|X_{P}|_{\bar{a},\bar{p}}\le\gamma\mu,\\
|X_{P}|^{{\mathcal {L}}_{\bar{a},\bar{p}}}\le M\mu,
\end{eqnarray*}
for all $(x,y,z,\bar{z})\in D_{a,p}(r,s),~\xi\in\Pi$. Then there
exits a Cantor set $\Pi_{\gamma}\subset\Pi$ with
$|\Pi\setminus\Pi_{\gamma}|\rightarrow 0$, as $\gamma\rightarrow 0$,
and a family of $ C^{2}$ symplectic transformations
$$
\Psi_{\xi}:
D_{a+\frac{\bar{a}-a}{2},p}(\frac{r}2,\frac{s}2)\rightarrow D_{
a,p}(r,s), ~~\xi\in \Pi_{\gamma},
$$
which are Lipschitz continuous in parameter $\xi$ and $ C^{2}$
uniformly close to the identity, such that for each
$\xi\in\Pi_{\gamma}$, corresponding to the unperturbed torus
$T_{\xi}$, the associated perturbed invariant torus can be described
as
$$
\left\{\begin{array}{ll}
\dot x & =  \omega_{*}(\xi),\\
\dot y & =  0,\\
\dot z & =\Omega^{*}(\xi)\bar{z},\\
\dot{\bar{z}}&=-\Omega^{*}(\xi)z,
\end{array}\right.
$$
where
\begin{eqnarray}
|\omega_{*}-\omega|+\frac{\gamma}{M}|\omega_{*}-\omega|^{{\mathcal
{L}}}\leq c\mu_*,\nonumber\\
|\Omega^{*}-\Omega|+\frac{\gamma}{M}|\Omega^{*}-\Omega|_{-\delta}^{{\mathcal
{L}}}\leq c\mu_*.\nonumber
 \end{eqnarray}
The perturbation $P^{*}=P\circ\Phi_{\xi}$ is real analytic on
variables $x$, $C^{2}$ on variables $(y,z,\bar{z})$ Lipshchtz
continuous on the parameters, and
$$
X_{P^{*}}|_{(y,z,\bar{z})=\{0\}}=0,
$$
for all $x\in {\bf T}^{n},~\xi\in\Pi_{\gamma}$. Namely, the
unperturbed torus $T_{\xi}={\bf
T}^n\times\{0\}\times\{0\}\times\{0\}$ associated to the frequency
$\omega(\xi)$ and $z\bar{z}$-space frequency $\Omega(\xi)$ persists
and gives rise to an analytic, Diophantine, invariant torus of the
perturbed system with the frequency $\omega_{*}(\xi)$ and
$z\bar{z}-$space frequency $\Omega^{*}(\xi)$. Moreover, these
perturbed tori form a Lipshcitz continuous family. }

\vspace{0.5cm} \noindent {\bf Remark I} When the infinite spectrum
$\Omega$ is multiple,  we first give some notations.

For given $\rho\in N_+$, and
$\{\imath_1,\cdots,\imath_n\}\subset{\bf Z}^{\rho}$, denote
$Z_{1}^{\rho}={\bf Z}^{\rho}\setminus\{\imath_1,\cdots,\imath_n\}$,
and without loss of generality, set
$0\in\{\imath_1,\cdots,\imath_n\}$.

Throughout this paper, we set
$$
{\mathcal {L}}^{a,p}=\{z:~z=(\cdots,z_i,\cdots)_{i\in
Z_{1}^{\rho}}\}
$$
with the norm
\begin{equation}\label{norm1}
|z|^{a,p}=\sqrt{\sum_{i\in Z_{1}^{\rho}}|z_i|^{2}|i|^{2p}{\rm
e}^{2|i|a}},
\end{equation}
where $i=(i_1,\cdots,i_{\rho}),~|i|=|i_1|+\cdots+|i_{\rho}|$, for
fixed $a\ge0,p\ge0$. For $\rho=1$, we set $Z_{1}^{1}=\{1,2,\cdots\}$
and the corresponding norm is $|z|^{a,p}=\sqrt{\sum_{i\ge
1}|z_i|^{2}i^{2p}{\rm e}^{2ai}}.$

Consider the function family ${\mathcal
{F}}_{a,p}^{\bar{a},\bar{p}}$. Let $F: {\mathcal
{P}}^{a,p}\rightarrow {\bf R}$, where ${\mathcal
{P}}^{a,p}=\{(x,y,z,\bar{z})\in {\bf T}^{n}\times {\bf R}^{n}\times
{\bf \mathcal {L}}^{a,p}\times {\bf \mathcal {L}}^{a,p}\}$. We say
$F\in{\mathcal {F}}_{a,p}^{\bar{a},\bar{p}}$, if $F$ has the
following properties:
\begin{itemize}
\item [{\rm 1)}]  {\it F} is a real analytic function in
$(x,y,z,\bar{z})$ and Lipschitz continuous in parameter $\xi$;
 \item
[{\rm 2)}] For fixed $\xi$, $F$ can be expanded as:
$$
F(x,y,z,\bar{z},\xi)=\sum_{k\in {\bf Z}^{n},m\in{\bf
N}^{n},q,\bar{q}}F_{kmq\bar{q}}(\xi)y^{m}z^{q}\bar{z}^{\bar{q}}{\rm
e}^{\sqrt{-1}\langle k,x\rangle},
$$
where $q=(\cdots,q_i,\cdots)_{i\in
Z_{1}^{\rho},}~\bar{q}=(\cdots,\bar{q}_i,\cdots)_{i\in
Z_{1}^{\rho}}$ have finitely many non-zero positive components; for
the case $\rho=1$, we set $q=(q_1,q_2,\cdots)$. The term
$z^{q}\bar{z}^{\bar{q}}$ denotes
$\prod_{i}z_i^{q_i}\bar{z}_{i}^{\bar{q}_i}$;
\item [{\rm 3)}]  $X_{F}$ is finite in the following weighted norm:
\begin{eqnarray}
|X_{F}|^{\bar{a},\bar{p}}_{s,D_{a,p}(r,s)\times\Pi}&=&|F_{y}|_{D_{a,p}(r,s)\times\Pi}
+\frac{1}{s^{2}}|F_{x}|_{D_{a,p}(r,s)\times\Pi}\nonumber\\
&+&\frac{1}{s}|F_{\bar{z}}|^{\bar{a},\bar{p}}_{D_{a,p}(r,s)\times\Pi}
+\frac{1}{s}|F_{z}|^{\bar{a},\bar{p}}_{D_{a,p}(r,s)\times\Pi}\nonumber\\
&<& \infty,\nonumber
\end{eqnarray}
where $\bar{a},~\bar{p}$ are fixed positive numbers;
\item [{\rm
4)}]  The Lipschitz semi-norm
$$
|X_{F}|^{{\mathcal
{L}}_{\bar{a},\bar{p}}}_{s,D_{a,p}(r,s)}=\sup_{\xi\neq\varepsilon}
\frac{|\bigtriangleup_{\xi\varepsilon}X_{F}|^{\bar{a},\bar{p}}_{s,D_{a,p}(r,s)}}{|\xi-\varepsilon|}<\infty,
$$
where
$\bigtriangleup_{\xi\varepsilon}=X_{P}(\cdot,\xi)-X_{P}(\varepsilon,\cdot)$,
and the supremum is taken over $\Pi$.
\end{itemize}

Consider the following infinite dimensional Hamiltonian system with
small perturbation
\begin{eqnarray}
H=\langle\omega(\xi),y\rangle+\sum_{i\in
Z_{1}^{\rho}}\Omega_{i}(\xi)z_i\bar{z}_i
+P(x,y,z,\bar{z};\xi),\label{1.1}
\end{eqnarray}
where $(x,y,z,\bar{z})$ lies in the complex neighborhood
$$D_{a,p}(s,r)=\{(x,y,z,\bar{z}):|{\rm Im}~x|<r, |y|<s^2,
|z|^{a,p},|\bar{z}|^{a,p}<s\}$$of ${\bf T}^{n}\times \{0\}\times
\{0\}\times \{0\}$ with the norm (\ref{norm1}). The frequencies
$\omega$ and $\Omega$ depend on $\xi\in\Pi\subset {\bf R}^{n}$,
$\Pi$ is a closed bounded region.

With respect to the symplectic form
$$
\sum_{i=1}^{n}{\rm d}x_{i}\wedge {\rm d}y_{i}+\sum_{i\in
Z_{1}^{\rho}}{\rm d}z_{i}\wedge {\rm d}\bar{z}_{i},
$$
the associated unperturbed motion of {\rm(\ref{1.1})} is simply
described by
$$
\left\{\begin{array}{llll}
\dot x & =  \omega(\xi),\\
\dot y & =  0,\\
\dot z_{i} & =-\Omega_{i}(\xi)\bar{z}_{i},\\
\dot{\bar{z}}_{i}&=\Omega_{i}(\xi)z_i.
\end{array}\right.
$$
Hence, for each $\xi\in\Pi$, there is a family of invariant
$n$-dimensional torus ${\mathcal {T}}_{\xi}={\bf
T}^n\times\{0\}\times\{0\}\times\{0\}$ with fixed frequency
$\omega(\xi)$.

We also make the following assumptions for (\ref{1.1}).
\begin{itemize}
\item[{\bf A1)'}] {\it Nondegeneracy.}~~The map
$\xi\rightarrow\omega(\xi)$ is a  homemorphism between $\Pi$ and its
image, and Lipschtiz continuous in both directions.
\item[{\bf A2)'}] {\it Spectral Asymptotics.}~~There exist $d>0$ and
$\delta<0$ such that for all $i\in Z_{1}^{\rho}$,
\begin{eqnarray}
\Omega_{i}(\xi)&\ne&0,~~i\in Z_{1}^{\rho},~~\xi\in\Pi,\\
\Omega_{i}(\xi)&=&\bar{\Omega}_{i}+\tilde{\Omega}_i,~~\bar{\Omega}_{i}=|i|^{d}+o(|i|^{d}),~\tilde{\Omega}_i=o(|i|^{\delta}),
\end{eqnarray}
and
\begin{equation}
\bar{\Omega}_{i}-\bar{\Omega}_{j}=|i|^d-|j|^d+o(|j|^{-\delta}),~~|j|\le
|i|.
\end{equation}
Moreover, $\omega(\xi),~\Omega(\xi)$ are Lipschtiz continuous in
$\xi$, and

\begin{equation}\label{omega}
|\omega|^{\mathcal {L}}_{\Pi}+|\Omega|^{\mathcal
{L}}_{-\delta,\Pi}\leq M<\infty,
\end{equation}
where
\begin{eqnarray*}
|\omega|_{\pi}^{\mathcal
{L}}&=&\sup_{\xi\ne\varepsilon}\frac{|\omega(\xi)-\omega(\varepsilon)|}{|\xi-\varepsilon|},\\
|\Omega|^{\mathcal
{L}}_{-\delta,\Pi}&=&\sup_{\xi\ne\varepsilon}\sup_{|i|\ge
1}|\frac{\Omega_i(\xi)-\Omega_i(\varepsilon)||i|^{-\delta}}{|\xi-\varepsilon|}.
\end{eqnarray*}
\item[{\bf A3)'}]{\it Regularity.}~~The perturbation $P\in{\mathcal
{F}}_{a,p}^{\bar{a},\bar{p}}$, that is, $P$ is real analytic in the
space coordinate and Lipschitz in the parameters, and for each
$\xi$, Hamiltonian vector space field
$X_{P}=(P_{y},-P_{x},P_{\bar{z}},-P_{z})^{T}$ defines a real
analytic map
$$
X_{P}:~{\mathcal {P}}^{a,p}\rightarrow{\mathcal
{P}}^{\bar{a},\bar{p}},
$$
where $\bar{a}>a>0,~\bar{p}\ge p\ge0$.
\item[{\bf A4)'}]{\it Special form of the perturbation.}~~The
perturbation $P\in{\mathcal {A}}$ where
$$
{\mathcal {A}}=\{P\in {\mathcal
{F}}_{a,p}^{\bar{a},\bar{p}}:~P_{kmq\bar{q}}=0
{\rm~if~}\sum_{j=1}^{n}k_j{\imath}_j+\sum_{i\in
Z_{1}^{\rho}}(q_{i}-\bar{q}_i)i\ne0\}.
$$
\end{itemize}

 \vspace{0.5cm}
\noindent{\bf Theorem  A'.} {\it Suppose that (\ref{1.1}) satisfies
 ${\rm A1)'-A4)'}$. For given $r,s>0$, if there exits a sufficiently small
$\mu=\mu(a,p,\bar{p},r,s,n,\tau)>0$, or equivalently
$\mu_*=\mu_*(a,p,\bar{p},r,s,n,\tau)>0$, such that
\begin{eqnarray}
&~&|X_{P}|^{\bar{a},\bar{p}}\leq \gamma\mu,\label{xp}\\
&~&|X_{P}|^{{\mathcal {L}}_{\bar{a},\bar{p}}}\leq M\mu,\label{xpl}
\end{eqnarray}
for all $(x,y,z,\bar{z})\in D_{a,p}(r,s),~\xi\in\Pi$. Then there
exit a Cantor set $\Pi_{\gamma}\subset\Pi$ with
$|\Pi\setminus\Pi_{\gamma}|\rightarrow 0$, as $\gamma\rightarrow 0$,
and a family of $ C^{2}$ symplectic transformations
$$
\Psi_{\xi}: D_{a+\frac
{\bar{a}-a}{2},p}(\frac{r}2,\frac{s}2)\rightarrow D_{a,p}(r,s),
~~\xi\in \Pi_{\gamma},
$$
which are Lipschitz continuous in parameter $\xi$ and $ C^{2}$
uniformly close to the identity, such that for each
$\xi\in\Pi_{\gamma}$, corresponding to the unperturbed torus
$T_{\xi}$, the associated perturbed invariant torus can be described
as
$$
\left\{\begin{array}{llll}
\dot x & =  \omega_{*}(\xi),\\
\dot y & =  0,\\
\dot z_{i} & =-\Omega^{*}_{i}(\xi)\bar{z}_{i},\\
\dot{\bar{z}}_{i}&=\Omega^{*}_{i}(\xi)z_i,
\end{array}\right.
$$
where
\begin{eqnarray}
&~&|\omega_{*}-\omega|+\frac{\gamma}{M}|\omega_{*}-\omega|^{{\mathcal
{L}}}\leq c\mu_*,\nonumber\\
&~&|\Omega^{*}-\Omega|_{-\delta,\Pi_{\gamma}}
+\frac{\gamma}{M}|\Omega^{*}-\Omega|_{-\delta,\Pi_{\gamma}}^{{\mathcal
{L}}}\leq c\mu_*.\nonumber
 \end{eqnarray}
Moreover the perturbation $P^{*}=P\circ\Phi_{\xi}$ is real analytic
in phase variables $x$, $C^2$ in variables $(y,z,\bar{z})$,
Lipshchtz continuous in the parameters, and
$$
X_{P^{*}}|_{(y,z,\bar{z})=\{0\}}=0,
$$
for all $x\in {\bf T}^{n},~\xi\in\Pi_{\gamma}$. Namely, the
unperturbed torus $T_{\xi}={\bf
T}^n\times\{0\}\times\{0\}\times\{0\}$ associated to the frequency
$\omega(\xi)$  persists and gives rise to an analytic, Diophantine,
invariant torus of the perturbed system with the frequency
$\omega_{*}(\xi)$. Moreover, these perturbed tori form a Lipshcitz
continuous family. }

The first three ones are the same as $A1)-A3)$, and the fourth one
is introduced in \cite{geng}. The Hamiltonian systems deriving from
the PDEs in high dimension not containing explicitly the space and
the time variables do have the special form.  A4)'  make the measure
estimates simpler, and  for the non-multiple case, A4)' is not
needed as in Theorem A.

\vspace{0.5cm}\noindent{\bf Remark II}~ The case $d\ge 2,~\rho>1$ or
$d>1,~\rho=1$ has been proved in \cite{posel1}.
According to the weaker spectral asymptotics, we consider the
following modified Diophantine condition :
\begin{eqnarray}\label{dufantu}
\{\xi:|\langle k,\omega(\xi)\rangle+\langle l,\Omega(\xi)\rangle|\ge
\frac{\gamma}{ A_{k,l}},~(k,l)\in{\mathcal {Z}}\}.
\end{eqnarray}
For details,
\begin{eqnarray}
&~&|\langle
k,\omega(\xi)\rangle|\ge\frac{\gamma}{1+|k|^{\tau}},~~\xi\in\Pi,~~|l|=0,\nonumber\\
&~&|\langle k,\omega(\xi)\rangle+\langle
l,\Omega(\xi)\rangle|\ge\frac{\gamma\langle
l\rangle_{d}}{1+|k|^{\tau}},
~~\xi\in\Pi,~~l\in\Lambda_+,\nonumber\\
&~&|\langle k,\omega(\xi)\rangle+\langle
l,\Omega(\xi)\rangle|\ge\frac{\gamma}{|i|^{c(\rho)}(1+|k|^{\tau})},
~~\xi\in\Pi,~l\in \Lambda_-, \nonumber
\end{eqnarray}
where $l\in\Lambda_-$ denotes that $l$ contains two opposite
components, $\Lambda_+$ denotes others expect $|l|=0$.  And $i$ is
the site of the positive components of $l$,
 $c(\rho)\ge c_1(\rho)+\rho$, $\tau\ge n+\frac
{c(\rho)+2}{d}+4$, $c_1(\rho)$ is constant only dependeding on
$\rho$,~and $\langle l\rangle_{d}=|\sum_{i} l_i|i|^{d}|$. When
$\rho=1$, we set $c(\rho)=\frac 52$ and for $l\in\Lambda_-$, we have
$$
|\langle k,\omega(\xi)\rangle+\langle
l,\Omega(\xi)\rangle|\ge\frac{\gamma}{|i|^{\frac52}(1+|k|^{\tau})},
~~\xi\in\Pi,~l\in \Lambda_-.
$$

One can find, when $l\in\Lambda_{-}$, the frequencies satisfying the
modified Diophantine condition are more than before and the the
other case is the same. The definition is intrinsical. Roughly
speaking, the parameter $\frac{\gamma}{A_{k,l}}$ describes the
measure of the removed region for fixed $(k,l)$, and the spectral
asympotics describe the cardinal number of the  eliminated regions
for fixed $k$. Since the spectrum can't be separable, the cardinal
number of the removed regions get more, naturally, the measure of
the removed region must get smaller to guarantee the measure
$|\Pi\setminus\Pi_{\gamma}|\rightarrow0$ as $\gamma\rightarrow0$.

\vspace{0.5cm} \noindent{\bf Remark III} ~In the following, one can
find that we only need a decreasing sequence
$\{a\}_{\nu=0}^{\infty}$ such that
$$
a< a+\frac{\bar{a}-a}{2}\leq a_{\nu+1}<a_{\nu}<\bar{a},
$$
for all $\nu=0,1,\cdots$. Without loss of generality,  we change A3)
to
$$
X_{P}:~{\mathcal {P}}^{\frac{a}{4},p}\rightarrow{\mathcal
{P}}^{a,\bar{p}},
$$
where $0<a<\bar{a},~0\le p\le\bar{p}$ for simplicity.

This  paper is organized as follows. The section 2 produces the KAM steps where we
introduce a new diophantine conditions and truncate the normal
variables $z$ and $\bar{z}$. In section 3, we derive an iteration
Lemma. And in section 4, we give measure estimates and complete the
proof of the main theorem. In the final section, we apply our results to nonlinear
Schr\"{o}dinger equations and the Klein-Gordon equations with exponential nonlinearity, which possess weaker spectral asymptotics. Then we obtain the existence of quasi-periodic solutions for both
the Klein-Gordon equations with exponential nonlinearity and Schr\"{o}dinger equations in higher dimension.

\section{KAM steps}
 In this section, we outline an iterative
scheme for the Hamiltonian (\rm \ref{yiweihamidun}) ( and
(\ref{1.1})) in one KAM cycle, say, from
 $\nu$th KAM step to the $(\nu+1)$th-step.
Set
\begin{eqnarray}
  &&r_0=r,~\gamma_0=\gamma,~\mu_{0}=\mu,~M_{0}=M,\Pi_{0}=\Pi,\nonumber\\
  &&\omega_{0}=\omega,~\Omega_{0}=\Omega,~,H_{0}=H,~a_0=a,\nonumber\\
&&N_0=\langle\omega(\xi),y\rangle+\sum_{i\in
Z_{1}^{\rho}}\Omega_{i}(\xi)z_i\bar{z}_i,~P_0=P,\nonumber
\end{eqnarray}
where $H_{0}$ satisfies the assumptions ${\rm A1)-A3)}$ for Theorem
A (or A1)'-A4)' for Theorem A')  for all $(x,y,z,\bar{z})\in
D_{\frac a4,p}(r_{0},s_{0}),~\xi\in\Pi_{0}.$ As usual, we construct
a sympletic transformation
$$
\Phi_{1}:~D_{\frac a4,p}(r_{1},s_{1})\times\Pi_{1}\rightarrow
D_{\frac a4,p}(r_{0},s_{0})\times\Pi_{1}
$$
where $r_{1}<r_{0},s_{1}<s_{0}$. Our aim is to make the perturbation
much smaller on a smaller domain $D_{\frac
a4,p}(r_{1},s_{1})\times\Pi_1$, which will be defined clearly below.

 We also have a sequence of symplectic transformations and domains as
 follows
$$
\Phi_{i}:~D_{\frac a4,p}(r_{i},s_{i})\times\Pi_{i}\rightarrow
D_{\frac a4,p}(r_{i-1},s_{i-1})\times\Pi_{i}
$$
for all $i=1,\cdots,\nu$. And we arrive at the following real
analytic Hamiltonian:
 \begin{eqnarray}
 H_{\nu}&=&H\circ\Phi_{0}\circ\Phi_{1}\circ\cdots\circ\Phi_{\nu}=N_{\nu}+P_{\nu},\label{h}\\
 N_{\nu}&=&e_{\nu}(\xi)+\langle\omega_{\nu}(\xi),y\rangle+
\sum_{n\in Z_{1}^{\rho}}\Omega_{i}(\xi)z_i\bar{z}_i,\label{n}
 \end{eqnarray}
 where
 \begin{eqnarray}
&~&|X_{P_{\nu}}|^{a_{\nu},\bar{p}}\leq
\gamma_{\nu}\mu_{\nu},~~|X_{P_{\nu}}|^{{\mathcal
{L}}_{{a_{\nu},\bar{p}}}}\leq
M_{\nu}\mu_{\nu}, \label{xpnu}\nonumber\\
&~&|\omega_{\nu}|^{{\mathcal {L}}}_{\Pi}+|\Omega_{\nu}|^{{\mathcal
{L}}}_{-\delta,\Pi}\leq M_{\nu}<\infty,\nonumber\\
\end{eqnarray}
for all $(x,y,z,\bar{z})\in D_{\frac a4,p}(r_{\nu},s_{\nu})$ and
$\xi\in\Pi_{\nu}\subset\Pi_0.$

In the following, we try to find a symplectic transformation
$$
\Phi_{\nu+1}:D_{\frac
a4,p}(r_{\nu+1},s_{\nu+1})\times\Pi_{\nu+1}\rightarrow D_{\frac
a4,p}(r_{\nu},s_{\nu})\times\Pi_{\nu+1},
$$
such that  $P_{\nu+1}$ is the new perturbation with the similar
estimates on $D_{\frac a4,p}(r_{\nu+1},s_{\nu+1})\times\Pi_{\nu+1}$.



 For simplicity, we shall omit index for all
quantities of the present KAM step (the $\nu$th-step) and index all
quantities (Hamiltonian, normal form, perturbation, transformation,
and domains, etc) in the next KAM step (the $(\nu+1)$-th step) by
$``+"$. All constants $c_i, c$ below are positive and independent of
the iteration process. To simplify the notations, we shall omit the
scripts like $a,~\bar{p},~r,~D(r,s)$, and only mark the changes.

Define
\begin{eqnarray}
r_+&=&\frac r2+\frac {r_0}4,\nonumber\\
a_+&=&\frac{a}{2}+\frac{a_{0}}4,\nonumber\\
\gamma_+&=&\frac\gamma 2+\frac{\gamma_0}4,\nonumber\\
M_{+}&=&\frac{M}{2}+M_{0},\nonumber\\
s_+&=&\eta s,~~\eta=\mu^{\frac{1}{3}},~~\mu=s^{\frac12},\nonumber\\
K_+&=&([\log\frac{1}{\mu}]+1)^{3\alpha_1},~~I_+=([\log\frac{1}{\mu}]+1)^{3\alpha_2},\nonumber\\
D_{\frac a4,p;~i}&=& D_{\frac a4,p}(r_++\frac{i-1}{8}(r-r_+),i\eta
s),~~
i=1,2,\cdots,8,\nonumber\\
D(\beta)&=&\{y\in C^d: |y|<\beta\},~~~\beta>0,\nonumber\\
\hat D_{\frac a4,p}(\beta)&=&D_{\frac a4,p}(r_++\frac{7}{8}(r-r_+),\beta),~~~\beta>0,\nonumber\\
\tilde D_{\frac a4,p}(\beta)&=&D_{\frac a4,p}(r_++\frac34(r-r_+),\frac \beta2),~~~\beta>0,\nonumber\\
D_+&=&D_{\frac a4,p;~1},~~D=D(r,s),\nonumber\\
 \Pi_{+}&=&\{ \xi\in \Pi: |\langle k,\omega(\xi)\rangle+\langle
l,\Omega(\xi)\rangle|>\frac{\gamma}{A_{k,l}}\},\nonumber\\
\Gamma(r-r_+)&=&\sum_{0<|k|\le K_+}|k|^{4\tau+4} {\rm
e}^{-|k|\frac{r-r_+}{8}},
\nonumber\\
C(a-a_+)&=&r\sum_{i\in Z_{1}^{\rho}}|i|^{2c(\rho)+2}{\rm
e}^{-(a-a_+)|i|},\nonumber
\end{eqnarray}
where $\alpha_1,~\alpha_2$ are constants to be specified.

\subsection{Truncation}
\noindent{}

We write  $P$ in the Taylor-Fourier series:
$$
P(x,y,z,\bar{z})=\displaystyle\sum_{m,q,\bar{q}}\displaystyle\sum_{k\in
Z^n}P_{kmq\bar{q}}y^{m}z^{q}{\bar{z}}^{\bar{q}} e^{\sqrt{-1}\langle
k,x\rangle},
$$
where $k\in{\bf Z}^{n},~m\in{\bf N}^{n}$, $q,~\bar{q}$ are
multi-index infinite vectors with finitely many nonzero components
of positive integers.

 We take $K_+,~I_+$ such that
\begin{itemize}
\item[{\rm\bf H1)}]$\displaystyle\int^\infty_{K_+}t^{n+4}
 e^{-t\frac{r-r_+}{16}}{\rm
d}t\leq \mu$,
\item[{\rm\bf H2)}]$\displaystyle\int^\infty_{I_+}t^{\rho+4}
 e^{-t(a-a_+)}{\rm
d}t\leq \mu$.
\end{itemize}
Then, in a usual manner  we have

\begin{lemma}
 $\forall~(m,q,\bar{q})$, $\xi\in\Pi_{0}$,
 there is a constant $c_{1}$ such that
\begin{eqnarray}
|X_{P-R}|^{a,\bar{p}}_{7\eta s,D_{\frac a4,p;7}}\leq
c_{1}\gamma\mu^{2},~~|X_{R}|^{a,\bar{p}}_{r,D}\leq \gamma\mu,\nonumber\\
|X_{P-R}|^{{\mathcal {L}}_{a,\bar{p}}}_{7\eta s,D_{\frac
a4,p;7}}\leq c_{1} M\mu^{2},~~|X_{R}|^{{\mathcal
{L}}_{a,\bar{p}}}_{r,D}\leq M\mu,\nonumber
\end{eqnarray}
where $R=R^{0}+R^{1}+R^{2}$ as follows:
\begin{eqnarray*}
R^{0}&=&\sum_{0<|k|\leq K_{+},|m|\le1}P_{km00}~y^{m}{\rm
e}^{\sqrt{-1}\langle
k,x\rangle},\\
R^{1}&=&\sum_{|k|\leq K_{+}}(\langle P_{k10},z\rangle+\langle
P_{k001},\bar{z}\rangle){\rm
e}^{\sqrt{-1}\langle k,x\rangle}\\
&=&\sum_{k,i}P^{k10}_{i}z_i{\rm e}^{\sqrt{-1}\langle
k,x\rangle}+\sum_{k,i}P^{K01}_{i}\bar{z}_i{\rm
e}^{\sqrt{-1}\langle k,x\rangle},\\
R^{2}&=&\sum_{|k|\leq K_{+}}(\langle P_{k020}~z,z\rangle+\langle
P_{k011}~z,\bar{z}\rangle+\langle
P_{k002}~\bar{z},\bar{z}\rangle){\rm
e}^{\sqrt{-1}\langle k,x\rangle}\\
&=&\sum_{|k|+||i|-|j||\ne0}P^{k20}_{ij}z_iz_j{\rm
e}^{\sqrt{-1}\langle
k,x\rangle}+\sum_{|k|+||i|-|j||\ne0}P^{k02}_{ij}\bar{z}_i\bar{z}_j{\rm
e}^{\sqrt{-1}\langle k,x\rangle}\\
&~&+ \sum_{|k|+||i|-|j||\ne0,|i|\le I_+}P^{k11}_{ij}z_i\bar{z}_j{\rm
e}^{\sqrt{-1}\langle k,x\rangle},
\end{eqnarray*}
and $P^{k10}_{i}=P_{k0q\bar{q}}$ with $q=e_{i},\bar{q}=0$;
$P^{k01}_{i}=P_{k0q\bar{q}}$ with $q=0,\bar{q}=e_{i}$;
$P^{k11}_{ij}=P_{k0q\bar{q}}$ with $q=e_{i},\bar{q}=e_{j}$;
$P^{k20}_{ij}=P_{k0q\bar{q}}$ with $q=e_{i}+e_{j},\bar{q}=0$;
$P^{k02}_{ij}=P_{k0q\bar{q}}$ with $q=0,\bar{q}=e_{i}+e_{j}$; here
$e_{i}$ denotes the vector with the $i$-th component being 1 and
other components being zero.
\end{lemma}
\begin{proof} Let $P-R=I+\bar{I}+II$, where
\begin{eqnarray*}
I&=&\sum_{|k|\ge K_+}P_{kmq\bar{q}}y^{m}z^{q}{\bar{z}}^{\bar{q}}
e^{\sqrt{-1}\langle k,x\rangle}, \\
\bar{I}&=&\sum_{k}\sum_{2|m|+|q+\bar{q}|\ge2}P_{kmq\bar{q}}y^{m}z^{q}{\bar{z}}^{\bar{q}}
e^{\sqrt{-1}\langle k,x\rangle},\\
II&=&\sum_{k}\sum_{|i|\ge I_+}P^{11}_{ij}z_{i}\bar{z}_{j}.
\end{eqnarray*}
As previous, ${\rm \bf H1 )}$ guarantees
$$
|X_{I}|^{a_+,\bar{p}}_{7\eta
s,D_{\tilde{a},p;7}},~|X_{\bar{I}}|^{a_+,\bar{p}}_{7\eta
s,D_{\tilde{a},p;7}}\leq c_{1}\gamma\mu^{2},~~|X_{I}|^{{\mathcal
{L}}_{a_+,\bar{p}}}_{7\eta
s,D_{\tilde{a},p;7}},~|X_{\bar{I}}|^{{\mathcal
{L}}_{a_+,\bar{p}}}_{7\eta s,D_{\tilde{a},p;7}}\leq c_{1} M\mu^{2}.
$$
Then we only consider the term $II$ where
$P^{11}=\sum_{k,i,j}P^{k11}_{ij}{\rm e}^{\sqrt{-1}\langle
k,x\rangle},~P^{11}_{ij}=\sum{k}P^{k11}_{ij}{\rm
e}^{\sqrt{-1}\langle k,x\rangle}$.

Since  $P^{11}=\partial_{z}\partial_{\bar{z}}P$, by (\rm\ref {xp}),
we have
$$
\|P^{11}\|\leq \sup_{|z|_{\frac a4,p,}=1}|P^{11}z|_{a,\bar{p}}=
 \sup_{|z|_{\frac a4,p,}=s}\frac 1s|P^{11}z|_{a,\bar{p},D(s)}\le \gamma\mu,
$$
where $\|\cdot\|$ means operator norm.

For any $z^{j}=(\cdots,{\rm e}^{-\frac a4|j|}|j|^{-p},\cdots)$, that
is, all components are zero except for the j-th,
$$
P^{11}z^j=(\cdots,P^{11}_{ij}{\rm e}^{-\frac a4|j|}|j|^{-p},\cdots).
$$
Moreover, $|z^{j}|^{\frac a4,p}=1$ yields
\begin{eqnarray*}
(|P^{11}z^{j}|^{a,\bar{p}})^{2}&=&\sum_{i\in
Z_{1}^{\rho}}|P^{11}_{ij}{\rm e}^{-\frac
a4|j|}|j|^{-p}|^{2}{\rm e}^{2a|i|}|i|^{2\bar{p}}\\
&=&{\rm e}^{-\frac a4|j|}|j|^{-2p}\sum_{i\in
Z_{1}^{\rho}}|P^{11}_{ij}|^{2}{\rm e}^{2a|i|}|i|^{2\bar{p}}\leq
\gamma\mu,
\end{eqnarray*}
i.e. for all $i,j\in Z_{1}^{\rho}$, we have
\begin{equation}
|P^{11}_{ij}|\le \gamma\mu {\rm e}^{-a|i|}|i|^{-\bar{p}}{\rm
e}^{\frac a4|j|}|j|^{p}.\label{rij}
\end{equation}
Then any $|z|^{\frac a4,p}\le s,~i.e.~|z_j|\le s {\rm e}^{-\frac
a4|j|}|j|^{-p}$ yields
\begin{eqnarray}
\sum_{j\in Z_{1}^{\rho}}|P^{11}_{ij}z_{j}|^{2} &\le& \gamma^2\mu^2
{\rm e}^{-2a|i|}|i|^{-2\bar{p}}{\rm e}^{-2|k|r}\sum_{j\in Z_{1}^{\rho}}|z_j|^2{\rm e}^{2\frac a4|j|}|j|^{2p}\nonumber\\
&\le&\gamma^2\mu^2s^2 {\rm e}^{-2a|i|}|i|^{-2\bar{p}}.\label{rijzj}
\end{eqnarray}
Then
\begin{eqnarray}
(|II_{\bar{z}}|^{a_+,\bar{p}}_{D})&\le&\sqrt{\sum_{|i|\ge
I_+}\sum_{j\in Z_{1}^{\rho}}|P^{11}_{ij}z_{j}|^{2}{\rm e}^{2a_+i}i^{2\bar{p}}}\nonumber\\
&\le&s\gamma\mu\sum_{|i|\ge I_+}{\rm e}^{-(a-a_+)i}\le
s\gamma\mu\int_{I_+}^{\infty}t^{\rho}{\rm e}^{-(a-a_+)t}{\rm d}t\nonumber\\
&\le& s\gamma\mu^{2}.\nonumber
\end{eqnarray}
The same is for $II_{z}$ and the Lipshchitz semi-norm.
\qquad\end{proof}

For the proof of  Theorem A, H2) shows
$$
\int_{I_+}^{\infty}t^{5}{\rm e}^{-t(a-a_+)}{\rm d}t\le \mu,
$$
and the truncation $R^2$ has the form:
\begin{eqnarray}
R^2&=&\sum_{|k|\le K_+,i,j}P^{k20}_{ij}z_iz_j{\rm
e}^{\sqrt{-1}\langle k,x\rangle}+\sum_{|k|\le
K_+,i,j}P^{k02}_{ij}\bar{z}_i\bar{z}_j{\rm e}^{\sqrt{-1}\langle
k,x\rangle}\nonumber\\
&+& \sum_{|k|\le K_+,i\le I_+,j}P^{k11}_{ij}z_i\bar{z}_j{\rm
e}^{\sqrt{-1}\langle k,x\rangle}.\nonumber
\end{eqnarray}
We choose $\rho=1,~c(\rho)=\frac52$, the results and the proof also
hold.
\subsection{Linear Equations}
\noindent{}
 We induce the time-1 map $\phi_{F}^{1}|_{t=1}$ generated by the
Hamiltonian F:
\begin{eqnarray}
F&=&\sum_{0<|k|\le K_+,~|m|\le1}F_{km00}y^{m}{\rm
e}^{\sqrt{-1}\langle k,x\rangle}\\
&~&+\sum_{|k|\le K_+,i}(\langle F^{k10}_i,z_i\rangle+\langle
F^{k01}_i,\bar{z}_i\rangle){\rm
e}^{\sqrt{-1}\langle k,x\rangle}\nonumber\\
&~&+\sum_{|k|\le K_+,|k|+||i|-|j||\ne0}(\langle F^{k20}_{ij}z_i,z_j
\rangle+\langle F^{k02}_{ij}\bar{z}_i,\bar{z}_j\rangle){\rm
e}^{\sqrt{-1}\langle k,x\rangle}\nonumber\\
&~&+\sum_{|k|\le K_+,|k|+||i|-|j||\ne0,|i|\le I_+}\langle
F^{k11}_{ij}z_i,\bar{z}_j\rangle {\rm e}^{\sqrt{-1}\langle
k,x\rangle},\nonumber
\end{eqnarray}
where the coefficients dependeding on $\xi$ are to be determined.

Let
$$
[R]=\sum_{|m|+|q|=1}P_{0mqq}y^{m}z^q\bar{z}^q+\sum_{i,j}P^{11}_{0,ij}z_{i}\bar{z}_{j}.
$$
Substituting $F,R,[R]$ into
\begin{equation}
\{N,F\}+R-[R]=0,\label{linear equation}
\end{equation}
we have the following equations for $|k|<K_+,~i,j\in Z_{1}^{\rho}:$
\begin{eqnarray}
\langle k,\omega\rangle
F_{km00}&=&\sqrt{-1}P_{km00},~k\ne0,~|m|\le1,\\
(\langle
k,\omega\rangle-\Omega_i)F^{k10}_{i}&=&\sqrt{-1}P^{k10}_{i},\\
(\langle
k,\omega\rangle+\Omega_i)F^{k01}_{ij}&=&\sqrt{-1}P^{k01}_{i},\\
(\langle
k,\omega\rangle-\Omega_i-\Omega_j)F^{k20}_{ij}&=&\sqrt{-1}P^{k20}_{ij},~|k|+||i|-|j||\ne0,\\
(\langle
k,\omega\rangle+\Omega_i+\Omega_j)F^{k02}_{ij}&=&\sqrt{-1}P^{k02}_{ij},~|k|+||i|-|j||\ne0,\\
(\langle
k,\omega\rangle+\Omega_i-\Omega_j)F^{k11}_{ij}&=&\sqrt{-1}P^{k11}_{ij},~|k|+||i|-|j||\ne0,~|i|\le
I_+.\label{fkij}
\end{eqnarray}
For Theorem A', note that the term
$\sum_{|i|=|j|}P^{011}_{ij}z_{i}\bar{z}_{j}$ can be decomposed into
$\sum_{i\in Z_{1}^{d}}P^{k011}_{ii}z_{i}\bar{z}_{i}$ and
$\sum_{|i|=|j|,i\ne j}P^{k011}_{ij}z_{i}\bar{z}_{j}$. Since $P\in
{\mathcal {A}}$, the last term is absent.

Consider $\xi\in\Pi_+$ which satisfies (\ref{dufantu}), then we have

\begin{lemma} $F$ is uniquely determined by {\rm(\ref{linear equation})} and
there exits a constant $c_{2}$ such that on $\hat{D}_{\frac
a4,p}(s)\times\Pi_{+}$,
\begin{eqnarray}
|X_{F}|^{a_+,\bar{p}}
&\leq&c_{2}\mu\Gamma(r-r_{+})C(a-a_+),\label{xf}\\
|X_{F}|^{{\mathcal
{L}}_{a_+,\bar{p}}}&\leq&c_{2}\frac{M}{\gamma}\mu\Gamma(r-r_{+})C(a-a_+).
\label{xfl}\end{eqnarray} Moreover, $F\in{\mathcal {A}}$ in Theorem
A'.
\end{lemma}
\begin{proof} First, we consider the case $l\in\Lambda_{-}$, i.e.,
\begin{equation}
|\langle k,\omega(\xi)\rangle+\Omega_{i}-\Omega_{j}|\ge
\frac{\gamma}{(1+|k|^{\tau})|i|^{c(\rho)}},~|k|+||i|-|j||\ne0,~|i|\le
I_+.\label{dufantu1}
\end{equation}

Since  $P^{11}=\partial_{z}\partial_{\bar{z}}P$, by (\rm\ref {xp}),
we have
$$
\|P^{11}\|\leq \sup_{|z|_{\frac a4,p}=1}|P^{11}z|_{a,\bar{p}}=
 \sup_{|z|_{\frac a4,p}=s}\frac 1s|P^{11}z|_{a,\bar{p},D(s)}\le \gamma\mu,
$$
where $\|\cdot\|$ means operator norm.

Expanding $P^{11}$ into its Fourier series, we have
\begin{eqnarray*}
\|P^{k11}\|\leq c\|P^{11}\|\le c \gamma\mu {\rm e}^{-|k|r},
\end{eqnarray*}
for all $k\in{\bf Z}^{n}$.

By lemma 2.1, choose $z^{j}=(\cdots,{\rm e}^{-\frac
a4|j|}|j|^{-p},\cdots)$, that is, all components are zero except for
the $j$-th. For all $i,j\in Z_1^{\rho}$, we have
\begin{equation}
|P^{k11}_{ij}|\le \gamma\mu {\rm e}^{-a|i|}|i|^{-\bar{p}}{\rm
e}^{\frac a4|j|}|j|^{p}{\rm e}^{-|k|r}.\label{rkij}
\end{equation}
Then any $|z|^{\frac a4,p}\le s,~i.e.,~|z_j|\le s {\rm e}^{-\frac
a4|j|}|j|^{-p}$, yields
\begin{eqnarray}
\sum_{j\in Z^{\rho}_{1}}|P^{k11}_{ij}z_{j}|^{2} &\le& \gamma^2\mu^2
{\rm e}^{-2a|i|}|i|^{-2\bar{p}}{\rm e}^{-2|k|r}\sum_{j\in Z_{1}^{\rho}}|z_j|^2{\rm e}^{\frac a2|j|}|j|^{2p}\nonumber\\
&\le&\gamma^2\mu^2s^2 {\rm e}^{-2a|i|}|i|^{-2\bar{p}}{\rm
e}^{-2|k|r}.\label{rkijzj}
\end{eqnarray}
 Denote $\delta_{k,ij}\equiv\langle
k,\omega\rangle+\Omega_{i}-\Omega_{j}$. According to (\rm\ref{fkij})
and (\rm\ref{dufantu1}), we have
\begin{eqnarray}
F^{k11}_{ij}&=&\sqrt{-1}\frac{P^{k11}_{ij}}{\delta_{k,ij}},\nonumber\\
|F^{k11}_{ij}z_{j}|&\leq&
\frac{(1+|k|^{\tau})|i|^{c(\rho)}}{\gamma}|P^{k11}_{ij}z_{j}|.\nonumber
\end{eqnarray}
This together with (\ref{rkijzj}) implies that
\begin{eqnarray}
|\sum_{j\in Z_{1}^{\rho}}F^{k11}_{ij}z_{j}|^{2}&\le& \sum_{j\in
Z_{1}^{\rho}}|F^{k11}_{ij}z_{j}|^{2}\le
\frac{(1+|k|^{\tau})^{2}|i|^{2c(\rho)}}{\gamma^2}\sum_{j\in Z_{1}^{\rho}}|P^{k11}_{ij}z_{j}|^{2}\nonumber\\
&\le &(1+|k|^{\tau})^{2}\mu^{2}s^{2}{\rm e}^{-2|k|r}{\rm
e}^{-2a|i|}|i|^{-2\bar{p}+2c(\rho)},\nonumber
\end{eqnarray}
and
\begin{eqnarray}
(|F^{k11}z|^{a_+,\bar{p}}_{s,~D\times\Pi_{+}})^{2}&= &\sum_{i\in
Z_{1}^{\rho}}|\sum_{j\in Z_{1}^{\rho}}
F^{k11}_{ij}z_{j}|^{2} {\rm e}^{2a_+|i|}|i|^{2\bar{p}}\nonumber\\
&\leq&\mu^{2}s^{2}|k|^{2\tau}{\rm e}^{-2|k|r}C(a-a_+).\nonumber
\end{eqnarray}
Then we have
\begin{eqnarray}
\frac{1}{s}|F^{11}z|^{a_+,\bar{p}}_{s,~\hat{D}_{\frac
a4,p}(s)\times\Pi_{+}}&\leq&\frac{1}{s}\sum_{|k|\leq
K_{+}}|F_{k}^{11}z|^{a_+,\bar{p}}_{s,\hat{D}_{\frac a4,p}(s)\times\Pi_{+}}{\rm e}^{|k|(r_{+}+\frac{7}{8}(r-r_{+}))}\nonumber\\
&\leq&
c\mu\sum_{0<|k|\le K_{+}}(1+|k|^{\tau}){\rm e}^{-|k|\frac{r-r_{+}}{8}}C(a-a_+)\nonumber\\
&\leq&c~\mu\Gamma(r-r_{+})C(a-a_+).\nonumber
\end{eqnarray}

 The Lipschitz estimate almost follows the same idea. Fist, we
decompose the following into two parts:
\begin{eqnarray*}
\triangle
F^{11}_{k,ij}z_{j}&=&\frac{P^{k11}_{ij}(\xi)z_{j}}{\delta_{k,ij}(\xi)}
-\frac{P^{k11}_{ij}(\zeta)z_{j}}{\delta_{k,ij}(\zeta)}\\
&=&\frac{P^{k11}_{ij}(\xi)z_{j}-P^{k11}_{ij}(\zeta)z_{j}}{\delta_{k,ij}(\xi)}
+P^{k11}_{ij}(\zeta)z_{j}(\frac{1}{\delta_{k,ij}(\xi)}-\frac{1}{\delta_{k,ij}(\xi)})\\
&=&\frac{\triangle
P^{k11}_{ij}z_{j}}{\delta_{k,ij}(\xi)}+P^{k11}_{ij}(\zeta)z_{j}\triangle\delta^{-1}_{k,ij}.
\end{eqnarray*}
Since
$$
-\triangle\delta^{-1}_{k,ij}=\frac{\triangle\delta_{k,ij}}
{\delta_{k,ij}(\xi)\delta_{k,ij}(\zeta)} =\frac{\langle
k,\triangle\omega\rangle+\triangle\Omega_{i}-\triangle\Omega_{j}}
{\delta_{k,ij}(\xi)\delta_{k,ij}(\zeta)},
$$
 we obtain
\begin{eqnarray*}
|\triangle\delta^{-1}_{k,ij}|&\le&\frac{(1+|k|^{\tau})^{2}|i|^{2c(\rho)}}{\gamma^{2}}(|k||\triangle\omega|
+|\triangle\Omega|_{-\delta})\\
&\le&c\frac{(1+|k|^{\tau})^{2}(1+|k|)^2|i|^{2c(\rho)}}{\gamma^{2}}M|\xi-\varepsilon|.
\end{eqnarray*}

For fixed $\xi, \varepsilon$, $\bigtriangleup_{\xi\varepsilon}$ is a
bounded linear operator, in the operator norm below:
\begin{eqnarray}
\|\bigtriangleup_{\xi\varepsilon}P^{11}\|&=& \sup_{|z|_{\frac
a4,p}=1}|\bigtriangleup_{\xi\varepsilon}P^{11}z|_{a,\bar{p}}=
\sup_{|z|_{\frac a4,p}=s}\frac1s|\bigtriangleup_{\xi\varepsilon}P^{11}z|_{a,\bar{p},D}\nonumber\\
&\leq&M\mu|\xi-\varepsilon|,\nonumber
\end{eqnarray}
and for all $k$,
\begin{eqnarray}
\|\bigtriangleup_{\xi\varepsilon}P^{k11}\|\le
c\|\bigtriangleup_{\xi\varepsilon}P^{11}\|{\rm e}^{-|k|r}\le
M\mu|\xi-\varepsilon|{\rm e}^{-|k|r}.
\end{eqnarray}
Choose $z^{j}=(\cdots,{\rm e}^{-\frac a4|j|}|j|^{-p},\cdots)$.
Similar argument as above implies
\begin{equation}
|\triangle P_{ij}^{k11}|^{2}\leq c~ M^{2}\mu^{2}|\xi-\zeta|^{2}{\rm
e}^{-2a|i|}|i|^{-2\bar{p}}{\rm e}^{\frac a2|j|}j^{2p}{\rm
e}^{-2|k|r},\label{rkijl}
\end{equation}
then, we have
\begin{eqnarray*}
&~&|\sum_{j\in Z_{1}^{\rho}}\frac{\triangle
P_{ij}^{k11}z_j}{\delta_{k,ij}}|^2\nonumber\\&\le&\sum_{j\in
Z_{1}^{\rho}}\frac{|\triangle
P_{ij}^{k11}z_j|^2}{|\delta_{k,ij}|^2}+
\frac{(1+|k|^{\tau})^{2}|i|^{2c(\rho)}}{\gamma^{2}}
\sum_{j\in Z_{1}^{\rho}}|\triangle P_{ij}^{k11}z_j|^{2}\nonumber\\
&\le&c \frac{(1+|k|^{\tau})^{2}}{\gamma^2}{\rm
e}^{-2|k|r}M^2\mu^2s^2|\xi-\varepsilon|^2
{\rm e}^{-2a|i|}|i|^{2c(\rho)},\nonumber\\
\end{eqnarray*}
\begin{eqnarray}
&~&|\sum_{j\in
Z_{1}^{\rho}}P^{k11}_{ij}z_{j}\Delta\delta_{k,ij}^{-1}|^{2}\nonumber\\
&\le&\sum_{j\in
Z_1^{\rho}}|P^{k11}_{ij}z_{j}|^2|\Delta\delta_{k,ij}^{-1}|^2+ c\frac
{(1+|k|^{\tau})^{4}|k|^2}{\gamma^4}|\xi-\varepsilon|^2M^2 \sum_{j\in
Z_1^{\rho}}|P^{k11}_{ij}z_{j}|^2.\nonumber
\end{eqnarray}
Hence
\begin{eqnarray}
&~&(|\triangle F_{k}^{11}z|^{a_+,\bar{p}}_{s,D\times\Pi_{+}})^{2}
=\sum_{|i|\le I_+}|\sum_{j\in Z_1^{\rho}}(\frac{\triangle
P^{k11}_{ij}z_{j}}{\delta_{k,ij}}+R^{k11}_{ij}z_{j}\triangle\delta^{-1}_{k,ij})|^{2}
{\rm e}^{2a_+|i|}|i|^{2\bar{p}}\nonumber\\
&\leq& 2\sum_{|i|\le I_+}(|\sum_{j\in Z_1^{\rho}}\frac{\triangle
P^{k11}_{ij}z_{j}}{\delta_{k,ij}}|^{2}
 +|\sum_{j\in Z_1^{\rho}}|P^{k11}_{ij}z_{j}\triangle\delta^{-1}_{k,ij}|^{2})
 {\rm e}^{2a_+|i|}|i|^{2\bar{p}}\nonumber\\
 &\leq&~c\frac{(1+|k|^{\tau})^{4}|k|^2}{\gamma^2}M^2\mu^2s^2|\xi-\varepsilon|^2 {\rm e}^{-2|k|r}
 {\rm e}^{-2a|i|}|i|^{4c(\rho)}. \nonumber
\end{eqnarray}
Since $\sqrt{a+b}\leq\sqrt{a}+\sqrt{b},~~a,b>0$, we have
\begin{eqnarray}
\frac{1}{s}|F_{k}^{11}z|^{{\mathcal
{L}}_{a_+,\bar{p}}}_{s,~D\times\Pi_{+}}
&=&\sup_{\xi\neq\zeta}\frac{|\triangle
F_{k}^{11}z|^{a_+,\bar{p}}_{s,~D\times\Pi_{+}}}{|\xi-\zeta|}\nonumber\\
&\leq&c~\frac{1}\gamma \mu M(1+|k|^{\tau})^{2}|k|{\rm
e}^{-|k|r}C(a-a_+),\nonumber
\end{eqnarray}
and
\begin{eqnarray}
\frac{1}{s}|F^{11}z|^{{\mathcal
{L}}_{a_+,\bar{p}}}_{s,\hat{D}_{\frac
a4,p}(s)\times\Pi_{+}}&\leq&\frac{1}{s}\sum_{0<|k|\leq
K_{+}}|F_{k}^{11}z|^{{\mathcal {L}}_{a_+,\bar{p}}}
_{s,D\times\Pi_{+}}~{\rm e}^{|k|(r_{+}+\frac{7}{8}(r-r_{+}))}C(a-a_+)\nonumber\\
&\leq&
c~\frac{M}{\gamma}\mu\sum_{0<|k|<K_{+}}(1+|k|^{\tau})^{2}|k|{\rm e}^{-|k|\frac{r-r_{+}}{8}}C(a-a_+)\nonumber\\
&\leq&c~\frac{M}{\gamma}\mu\Gamma(r-r_{+})C(a-a_+).
\end{eqnarray}

The estimates of $|F^{11}z|^{a_+,\bar{p}},~|F^{11}z|^{{\mathcal
{L}}_{a_+,\bar{p}}}$ show all differences from the previous KAM
steps. Then we give the outline estimates of $F_{i}^{k10}$.
Since(\ref{xp}), we have
$$
\frac{1}{r}|P^{10}|^{a,\bar{p}}_{r,D(r,s)}\leq \gamma\mu,\nonumber
$$
and
$$
|P^{k10}|^{a,\bar{p}}_{r,D(r,s)}\leq c r\gamma\mu {\rm e}^{-|k|r},
$$
where $c$ is a constant. Then, according to the corresponding
Diophantine condition  $$|\langle k,\omega\rangle+\Omega_{j}|\geq
\frac{\gamma j^{d}}{A_{k}}\ge\frac{\gamma}{A_{k}}, $$and
$$
\sqrt{-1}F^{10}_{k,j}=\frac{R^{10}_{k,j}}{\langle
k,\omega\rangle+\Omega_{j}},~~~~~j\geq 1,
$$
the following hold:
\begin{eqnarray}
\frac{1}{r}|F^{10}z|^{a_+,\bar{p}}_{s,\hat{D}_{\frac a4,p}(s)\times\Pi_{+}}&\leq&c~\mu\Gamma(r-r_{+})C(a-a_+),\nonumber\\
\frac{1}{r}|F^{10}z|^{{\mathcal
{L}}_{a_+,\bar{p}}}_{s,\hat{D}_{\frac
a4,p}(s)\times\Pi_{+}}&\leq&c~\frac{M}{\gamma}\mu\Gamma(r-r_{+})C(a-a_+).\nonumber
\end{eqnarray}

All the other estimates are similar or even simpler and admit the
same results, for details, see \cite{geng}, \cite{posel1}. Moreover,
in Theorem A', $P\in{\mathcal {A}}$ implies $F\in{\mathcal {A}}$.

Above all, we fulfill the proof of the Lemma. \qquad\end{proof}

Let $\Phi_{+}$ denote the time-$1$ map generated by $F$. Then it is
a canonical transformation and
\begin{eqnarray}
H\circ\Phi_{+}&=&N+[R]+\int_{0}^{1}\{R_t,F\}\circ\phi^t_F{\rm
d}t+(P-R)\circ \Phi_{+}\nonumber\\
&=&N_{+}+P_{+},\nonumber
\end{eqnarray}
where $R_t=(1-t)\{[R],F\}+R$, and
\begin{eqnarray}
e_{+}&=&e+P_{0000},\label{e+}\\
\omega_{+}&=&\omega+P_{0100}\label{omega+},\\
\Omega_{+}&=&\Omega+P_{0011}\label{Omega+}.
\end{eqnarray}

For Theorem A, we induct $F$ as
\begin{eqnarray}\label{yiweiF}
F&=&\sum_{0<|k|\le K_+~|m|\le1}F_{k200}y^{m}{\rm
e}^{\sqrt{-1}\langle k,x\rangle}\\
&~&+\sum_{k,i}F^{10}_{k,i}z_i{\rm
e}^{\sqrt{-1}\langle k,x\rangle}+\sum_{k,i}F^{01}_{k,i}\bar{z}_i{\rm
e}^{\sqrt{-1}\langle k,x\rangle}\nonumber\\
&~&+\sum_{k,i,j}F^{20}_{k,ij}z_iz_j{\rm e}^{\sqrt{-1}\langle
k,x\rangle}+\sum_{k,i,j}F^{02}_{k,ij}\bar{z}_i\bar{z}_j{\rm
e}^{\sqrt{-1}\langle k,x\rangle}\nonumber\\
&~&+\sum_{k,j,i\le I_+}F^{11}_{k,ij}z_i\bar{z}_j{\rm
e}^{\sqrt{-1}\langle k,x\rangle},\nonumber
\end{eqnarray}
where
$F_{k200},~F_{k,i}^{10},~F_{k,i}^{01},~F^{20}_{k,ij},~F^{02}_{k,ij},~F^{11}_{k,
ij}$ are functions of $\xi$.

According to (\ref{linear equation}), we get the following equations
for $|k|\le K_+$:
\begin{eqnarray}
\langle k,\omega\rangle F_{km00}&=&\sqrt{-1}P_{km00},~~k\ne0,~~|m|\le1,\nonumber\\
(\langle k,\omega\rangle+\Omega_i)F^{k10_{i}}&=&\sqrt{-1}P^{k10}_{i},\nonumber\\
(\langle k,\omega\rangle-\Omega_i)F^{k01}_{i}&=&\sqrt{-1}P^{k01}_{i},\nonumber\\
(\langle k,\omega\rangle+\Omega_i+\Omega_j)F^{k20}_{ij}&=&\sqrt{-1}P^{k20}_{ij},\nonumber\\
(\langle k,\omega\rangle-\Omega_i-\Omega_j)F^{k02}_{ij}&=&\sqrt{-1}P^{k02}_{ij},\nonumber\\
(\langle k,\omega\rangle+\Omega_i-\Omega_j)F^{k11}_{ij}&=&\sqrt{-1}P^{k11}_{ij},~~i\le I_+,\nonumber\\
\end{eqnarray}
for all $\xi\in\Pi_+=\{\xi\in\Pi:~|\langle
k,\omega(\xi)\rangle+\langle
l,\Omega(\xi)\rangle|\ge\frac{\gamma}{A_{k,l}}\}$.

Replace $c(\rho)=\frac52,~\rho=1,~|i|=i,~C(a-a_+)=\sum_{0<i\le
I_+}i^{10}{\rm e}^{-(a-a_+)i}$, we have the lemma:

\begin{lemma}
 The linearized equation {\rm(\ref{linear equation})} has a unique
solution F normalized by $[F]=0$, and there exits a constant $c_{2}$
such that on $\hat{D}_{\frac a4,p}(s)\times\Pi_{+}$,
\begin{eqnarray}
|X_{F}|^{a_+,\bar{p}}
&\leq&c_{2}\mu\Gamma(r-r_{+})C(a-a_+),\nonumber\\
|X_{F}|^{{\mathcal
{L}}_{a_+,\bar{p}}}&\leq&c_{2}\frac{M}{\gamma}\mu\Gamma(r-r_{+})C(a-a_+).\nonumber
\end{eqnarray}
\end{lemma}
\subsection{Coordinate Transformation}
\noindent{}
\begin{lemma} Assume the following
\begin{itemize}
\item[{\rm \bf H3)}] $c_2\mu\Gamma (r-r_+) C(a-a_+)<
\frac{1}{8}(r-r_+)$;
\item[{\rm \bf H4)}] $c_2s^{2}\mu\Gamma (r-r_+)C(a-a+)<
5s_{+}^{2}$;
\item[{\rm \bf H5)}] $c_2s\mu\Gamma(r-r_{+})C(a-a+)<s_{+}$.
\end{itemize}
Then the following hold:
\begin{itemize}
\item[{\rm 1)}]
\begin{equation}
\phi^t_F:D_{\frac a4,p;3}\longrightarrow D_{\frac
a4,p;4},~\forall~0<t\leq 1,\label{2.24}
\end{equation}
so $\phi^{t}_{F}$ is well defined.
 \item[{\rm 2)}] $\Phi_{+}:D_{+}\longrightarrow D_{\frac a4,p}(r,s).$
\item[{\rm 3)}] There exists a positive constance $c_{3}$, such that
on $D_{\frac a4,p;3}\times\Pi_{+}$,
\begin{eqnarray}
|\phi^{t}_{F}-id|^{a_+,\bar{p}}&\leq& c_{3}
\mu\Gamma(r-r_+)C(a-a+),\nonumber\\
|D\phi^{t}_{F}-Id|^{a_+,\bar{p}}&\leq& c_{3}
\mu\Gamma(r-r_+)C(a-a+),\nonumber
\end{eqnarray}
and $\phi^{t}_{F}$ is Lipschitz continuous on parameter $\xi$, i.e.,
\begin{eqnarray}
|\phi^{t}_{F}-id|^{{\mathcal {L}}_{a_+,\bar{p}}}\leq c_{3}
\frac{M}{\gamma}\mu\Gamma(r-r_+)C(a-a+),\nonumber
\end{eqnarray}
for all $0\leq|t|\leq 1$.
\item[{\rm 4)}] On $D_{\frac a4,p;3}\times\Pi_+$
\begin{eqnarray}
|\Phi_+-id|^{a_+,\bar{p}}&\leq& c_{3}
\mu\Gamma(r-r_+)C(a-a+),\nonumber\\
|D\Phi_+-Id|^{a_+,\bar{p}}&\leq& c_{3}
\mu\Gamma(r-r_+)C(a-a+),\nonumber
\end{eqnarray}
and $\Phi_+$ is Lipschitz continuous on parameter $\xi$, i.e.,
\begin{eqnarray}
|\Phi_+-id|^{{\mathcal {L}}_{a_+,\bar{p}}}\leq c_{3}
\frac{M}{\gamma}\mu\Gamma(r-r_+)C(a-a+).\nonumber
\end{eqnarray}
\end{itemize}
\end{lemma}
\begin{proof}
\begin{itemize}
\item[{\rm 1)}]
Denote
$\phi_F^t=(\phi^t_{F1},\phi^t_{F2},\phi^t_{F3},\phi^t_{F4})^{\top}$,
where
 $\phi^t_{F1},\phi^t_{F2},\phi^t_{F3},\phi^t_{F4}$ are components in
 $x,y,z,\bar{z}$ planes respectively. Note that
\begin{equation}
\phi_F^t=id+\int_{0}^{t}X_{F}\circ\phi_{F}^{u}du.\label{phi}
\end{equation}

For any $(x,y,z,\bar{z})\in D_{\frac a4,p;3}$, let $t_*={\rm
\sup}\{t\in [0,1]: \phi_F^t(x,y,z,\bar{z}) \in D_{\frac a4_,p;4}\}$.
Since $0<\frac{a}{2}<a_+< a$,
$$
|w|^{\frac a4,p}\leq|w|^{a_+,\bar{p}},~~~\forall w\in
P^{\frac{a}{4},p},
$$
it follows from {\rm H2)-H5)} (\rm\ref{xf}) and (\rm\ref{xfl}), that
for any $0\le t\le t_*$,
\begin{eqnarray}
|\phi_{F1}^t(x,y,z,\bar{z})|_{D_{\frac
a4,p;3}}&\le&|x|+|\int_0^tF_y\circ\phi_F^u{\rm
d}u|\leq |x|+|F_y|_{\hat{D}_{\frac a4,p}\times\Pi_+} \nonumber\\
&\le& r_{+}+\frac{2}{8}(r-r_+)+c_{2}\mu\Gamma
(r-r_+)C(a-a+)\nonumber\\
&\le& r_++\frac{3}{8}(r-r_+),\nonumber\\
|\phi_{F2}^t(x,y,z,\bar{z})|_{D_{\frac
a4,p;3}}&\le&|y|+|\int_0^tF_x\circ\phi_F^u{\rm d}u|\leq
|y|+|F_x|_{\hat{D}_{\frac a4,p}\times\Pi_+} \nonumber\\
&\le& 9s_{+}^{2}+c_{2}s^{2}\mu\Gamma(r-r_{+})C(a-a+)\leq 16 s_{+}^{2},\nonumber\\
|\phi_{F3}^t(x,y,z,\bar{z})|^{\frac a4,p}_{D_{\frac
a4,p;3}}&\le&|z|^{\frac a4,p} +|\int_0^tF_{\bar{z}}\circ\phi_F^u{\rm
d}u|\nonumber\\&\le&~|z|^{\frac a4,p}+|F_{\bar{z}}|^{\frac a4,\bar{p}}_{\hat{D}_{\frac a4,p}\times\Pi_+}\nonumber\\
&\leq&|z|^{\frac a4,p}+|F_{\bar{z}}|^{a_+,\bar{p}}_{\hat{D}_{\frac
a4_{+},p}\times\Pi_+}\nonumber\\
&\leq& 3s_{+}+c_{2}s\mu\Gamma(r-r_+)C(a-a+)\le 4s_{+},\nonumber\\
|\phi_{F4}^t(x,y,z,\bar{z})|^{\frac a4,p}_{D_{\frac
a4,p;3}}&\le&|\bar{z}|^{\frac a4,p} +|\int_0^tF_{z}\circ\phi_F^u{\rm
d}u|\leq~|\bar{z}|^{\frac a4,p}+|F_{z}|^{a_+,\bar{p}}_{\hat{D}
\times\Pi_+}\nonumber\\
&\le& 3s_{+}+c_{2}s\mu\Gamma(r-r_+)C(a-a+)\le 4s_{+},\nonumber
\end{eqnarray}
i.e., $\phi_F^t(x,y,z,\bar{z})\in D_{\frac a4,p;4}$, $\forall 0\leq
t\leq t_{*}$. Thus, $t_*=1$ and (\rm\ref{2.24}) holds.
\item[2)] Since $\Phi_+=\phi_{F}^{1}$, the result follows from {\rm 1)}.
\item[3)] Using Lemma 2.2 and (\rm\ref{phi}), on $D_{\frac a4,p;3}\times\Pi_{+}$, we have
$$
|\Phi_+-id|^{a_+,\bar{p}} \leq|\int_{0}^{1}X_{F}\circ\phi_{F}^{u}du|
\leq|X_{F}|^{a_+,\bar{p}}_{\hat{D}_{\frac a4,p}\times\Pi_{+}}\leq
c_{3}\mu\Gamma(r-r_+)C(a-a+).
$$
Since
$$
D \Phi= Id +\int_{0}^{1}J(D^{2}F)D\phi_{F}^{t}dt,
$$
by Cauchy's inequality, we have
\begin{eqnarray*}
|D\Phi_+-Id|^{a_+,\bar{p}}\leq2|D^{2}F|\leq
c_{3}\mu\Gamma(r-r_+)C(a-a+).
\end{eqnarray*}
The Lipshitz semi-norm is the same as the first argument of  {\rm
3)}:
$$
|\Phi_+-id|^{{\mathcal {L}}_{a_+,p}}
\leq|\int_{0}^{1}X_{F}\circ\phi_{F}^{u}du| \leq|X_{F}|^{{\mathcal
{L}}_{{a_+},p}}_{\hat{D}_{\frac a4,p}\times\Pi_{+}}\leq
c_{3}\frac{M}{\gamma}\mu\Gamma(r-r_+)C(a-a+).
$$
The estimates on the higher-order partial derivatives of $\Phi$
follows a similar arguments.
\item[4)]4) is immediately follows from 3).
\end{itemize}
\qquad
\end{proof}

For the proof of Theorem A, we have the same hypotheses H3)-H5) with
$c(\rho)=\frac52$, then for Hamiltonian(\ref{yiweihamidun}) and the
coordinate transformation $\phi_{F}^{t}$ generated by the
Hamiltonian function (\ref{yiweiF}), lemma 2.4 also holds.

\subsection{The New Normal Form}
\noindent{}
\begin{lemma}
There is a constant $c_4$ such that on $\Pi_+$,
\begin{eqnarray}
&~&|\omega_+-\omega|_{\Pi_+}\leq
c_4\gamma\mu,~~~|\omega_+-\omega|^{\mathcal
{L}}_{\Pi_+}\leq c_4 M\mu,\nonumber\\
&~&|\Omega^+-\Omega|_{-\delta,\Pi_+}\leq
c_{4}\gamma\mu,~~~|\Omega^+-\Omega|^{{\mathcal
{L}}}_{-\delta,\Pi_+}\leq c_4 M\mu.\nonumber
\end{eqnarray}
\end{lemma}
\noindent{\bf Proof.} It immediately follows from (\rm\ref{xp}),
(\rm\ref{xpl}),(\rm\ref{rkij}),(\rm\ref{omega+}) and
(\rm\ref{Omega+}).
 Then,
\begin{eqnarray}
|\omega_+|^{\mathcal {L}}_{\Pi}+|\Omega^+|^{{\mathcal
{L}}}_{-\delta,\Pi}&\leq&|\omega|^{\mathcal
{L}}_{\Pi_+}+|\Omega|^{{\mathcal
{L}}}_{-\delta,\Pi_+}+|\omega_+-\omega|^{\mathcal
{L}}_{\Pi_+}+|\Omega^+-\Omega|^{{\mathcal
{L}}}_{-\delta,\Pi_+}\nonumber\\
&\leq& M+2c_4M\mu.\nonumber
\end{eqnarray}
\vspace{0.5cm}
\begin{lemma}
Assume that
\begin{itemize}
\item[{\rm \bf H6)}]
$2 c_4 M\mu\leq M_0-\frac{M}{2}$.
\end{itemize}
Then
$$|\omega_+|^{\mathcal {L}}_{\Pi}+|\Omega^+|^{{\mathcal
{L}}}_{-\delta,\Pi}\leq \frac{M}{2}+M_0 \leq M_+~.$$
\end{lemma}
\subsection{The New Frequency}
\noindent{}
\begin{lemma} Assume that
\begin{itemize}
\item[{\rm \bf H7)}]$c_{5}\mu(1+K_+)^{\tau}K_+I_+^{c(\rho)}\leq
\gamma-\gamma_+$.
\end{itemize}
Then, we have for all $|k|\le K_+,~|l|\le 2$,

\begin{eqnarray}
|\langle
k,\omega_+(\xi)\rangle|&\ge&\frac{\gamma}{1+|k|^{\tau}},~0<|k|\le
K_+,~l=0,\nonumber\\
|\langle k,\omega_+(\xi)\rangle+\langle l,\Omega^+(\xi)\rangle|&\ge&
\frac{\gamma_+\langle l\rangle_{d}}{1+|k|^{\tau}},~~~~|k|\le K_+,~l\in\Lambda_+,\nonumber\\
|\langle k,\omega_+(\xi)\rangle+\langle l,\Omega^+(\xi)\rangle|&\ge&
\frac{\gamma_+}{(1+|k|^{\tau})|i|^{c(\rho)}},~|k|\le
K_+,~l\in\Lambda_{-},~|i|<I_+,\nonumber
\end{eqnarray}
where $i$ is the site of the positive component of $l$.
\end{lemma}
\begin{proof}
Note that $\delta<0$. Denote
$\omega_+=\omega+\hat{\omega},~\Omega^+=\Omega+\hat{\Omega}$. We
have
$$
|\langle k,\hat{\omega}(\xi)\rangle+\langle
l,\hat{\Omega}(\xi)\rangle|\leq
|k||\hat{\omega}|_{\Pi_+}+|l|_{\delta}|\hat{\Omega}|^{a,\bar{p}}_{-\delta,\Pi_+}
\leq c_{5}\gamma\mu|k|\langle l\rangle_{d}.
$$
Combining with {\rm H7)}, as usual, we have
$$
c_{5}\mu(1+K_+)^{\tau}K_+\leq \gamma-\gamma_+,
$$
and
\begin{eqnarray*}
|k,\omega_+(\xi)\rangle+\langle l,\Omega^+(\xi)\rangle|
\geq\frac{\gamma_+\langle l\rangle_{d}}{1+|k|^{\tau}}.
\end{eqnarray*}
For $l\in\Lambda_-$,~$|l|_{\delta}\leq2$,
$$
|\langle k,\hat{\omega}(\xi)\rangle+\langle
l,\hat{\Omega}(\xi)\rangle|\leq
|k||\hat{\omega}|_{\Pi_+}+2|\hat{\Omega}|^{a,\bar{p}}_{-\delta,\Pi_+}
\leq c_{5}\gamma\mu|k|.
$$
Then we have
\begin{eqnarray}
|k,\omega_+(\xi)\rangle+\langle l,\Omega^+(\xi)\rangle|&\geq&
|\langle k,\omega(\xi)\rangle+\langle l,\Omega(\xi)\rangle|-|\langle
k,\hat{\omega}(\xi)\rangle+\langle
l,\hat{\Omega}(\xi)\rangle|\nonumber\\
&\geq&\frac{\gamma}{1+|k|^{\tau}|i|^{c(\rho)}}-
\frac{|k|(\gamma-\gamma_+)}
{(1+K_{+})^{\tau}K_+I_+^{c(\rho)}},\nonumber\\
&\geq&\frac{\gamma_+}{1+|k|^{\tau}|i|^{c(\rho)}}.\nonumber
\end{eqnarray}
\qquad\end{proof}

For the proof of Theorem A, we have

\noindent{ Lemma 2.7'} Assume that
\begin{itemize}
\item[{\rm \bf H7)'}]$c_{5}\mu(1+K_+^{\tau})K_+I_+^{\frac 52}\leq
\gamma-\gamma_+$.
\end{itemize}
Then we have that for all $|k|\le K_+,~|l|\le 2$,
\begin{eqnarray*}
|\langle
k,\omega_+(\xi)\rangle|&\ge&\frac{\gamma}{1+|k|^{\tau}},~0<|k|\le
K_+,~l=0,\nonumber\\
|\langle k,\omega_+(\xi)\rangle+\langle l,\Omega^+(\xi)\rangle|&\ge&
\frac{\gamma_+\langle l\rangle_{d}}{1+|k|^{\tau}},~~~~|k|\le K_+,~l\in\Lambda_+,\nonumber\\
|\langle k,\omega_+(\xi)\rangle+\langle l,\Omega^+(\xi)\rangle|&\ge&
\frac{\gamma_+}{(1+|k|^{\tau})i^{\frac52}},~|k|\le
K_+,~l\in\Lambda_{-},~0<i<I_+.\nonumber
\end{eqnarray*}

\subsection{The New Perturbation}

In this subsection,  we consider $|\cdot|^{a_+,\bar{p}}$, and we
denote $D_{\frac a4,p,i},~\tilde{D}_{\frac a4,p}(s),~\hat{D}_{\frac
a4,p}(s)$ by $D_{i},~\tilde{D},~\hat{D}$ respectively for short.

Recall the new perturbation
$$
P_+=\int_{0}^{1}\{R_t,F\}\circ\phi^t_F{\rm d}t+(P-R)\circ \Phi_{+},
$$
with $R_t=(1-t)[R]+tR$. Then the new perturbed vector field is
$$
X_{P_+}=X_{(P-R)\circ\phi_{F}^{1}}+
\int_{0}^{1}(\phi^{t}_{F})^{*}[X_{R_{t}},X_{F}]dt.
$$
We have
\begin{eqnarray}
|X_{(P-R)\circ\phi_{F}^{1}}|^{a_+,\bar{p}}_{s_+,D_{1}}&\leq&|X_{(P-R)}|^{a,\bar{p}}_{\frac{1}{2}\eta
s,D_{4}}\leq
c_6 \gamma\mu^{2},\label{p+1}\\
|X_{(P-R)\circ\phi_{F}^{1}}|^{{\mathcal
{L}}_{a_+,\bar{p}}}_{s_+,D_{1}}&\leq& |X_{(P-R)}|^{{\mathcal
{L}}_{a,\bar{p}}}_{\frac{1}{2}\eta s,D_{4}}\leq c_6
M\mu^{2}.\label{p+1l}
\end{eqnarray}
\begin{lemma}~ Let $\phi_{F}^{t}$ be the time-t map of the flow generated by
Hamiltonian $F$. Then there exits a positive constant $c$ such that
\begin{eqnarray}
|(\phi_{F}^{t})^{*}Y|^{a_+,\bar{p}}_{\frac{1}{4}\eta
s,D_{2}}&\leq&~c|Y|^{a_+,\bar{p}}_{\eta s,D_{4}},\nonumber\\
|(\phi_{F}^{t})^{*}Y|^{{\mathcal
{L}}_{a_+,\bar{p}}}_{\frac{1}{4}\eta s,D_{2}}&\leq&~c|Y|^{{\mathcal
{L}}_{a_+,\bar{p}}}_{\eta
s,D_{4}}+\frac{c}{(r-r_+)\eta^2}|Y|^{a_+,\bar{p}}_{\eta
s,D_{4}}|X_F|^{{\mathcal {L}}_{a_+,\bar{p}}}_{s,\hat{D}},\nonumber
\end{eqnarray}
for all $0\leq t\leq 1$.
\end{lemma}

 See {\rm section 3} of \rm\cite{posel1} for the detail proof.

In the following, we estimate the commutator $[X_{R_{t}},X_{F}]$ on
the domain:
$$
|[X_{R_{t}},X_{F}]|_{s}\leq |D X_{R}\cdot X_{F}|_{s}+|D X_{F}\cdot
X_{R}|_{s}.
$$
Using the generalized Cauchy estimate, {\rm lemma 3.1} and {\rm
lemma 3.2}, we obtain
\begin{eqnarray}
|[X_{R_{t}},X_{F}]|_{s,\tilde{D}}&\le&
(r-r_{+})^{-1}|X_P|_{s,D}|X_F|_{r,\hat{D}}\\
&\le& c (r-r_+)^{-1}\gamma\mu^{2}\Gamma(r-r_+)C(a-a_+).
\end{eqnarray}
Similarly,
\begin{eqnarray}
|[X_{R_{t}},X_{F}]|^{\mathcal {L}}_{s,\tilde{D}}&\leq& |D
X_{R}|^{\mathcal {L}}_{s,\tilde{D}}|X_{F}|_{s,\tilde{D}}+|D
X_{R}|_{s,\tilde{D}}|X_{F}|^{\mathcal {L}}_{s,\tilde{D}}\nonumber\\
&~&+|D X_{F}|^{\mathcal {L}}_{s,\tilde{D}}| X_{R}|_{s,\tilde{D}}+|D
X_{F}|_{s,\tilde{D}}|X_{R}|^{\mathcal {L}}_{s,\tilde{D}}\nonumber\\
&\leq& (r-r_+)^{-1}(|X_P|^{\mathcal
{L}}_{s,D}|X_F|_{s,\hat{D}}+|X_P|_{s,D}|X_F|^{\mathcal
{L}}_{s,D})\nonumber\\
&\leq& c(r-r_+)^{-1}M\mu^{2}\Gamma(r-r_+)C(a-a_+).\nonumber
\end{eqnarray}
Finally, we have $|Y|^{*}_{\eta s}\leq c\eta^{-2}|Y|^{*}_{s}$ for
any vector field $Y$, where $|\cdot|^{*}$ stands for either
$|\cdot|_{s}$ or $|\cdot|^{\mathcal {L}}_{s}$.

Combining the above, there exists a constant $c_{6}$ such that
\begin{eqnarray}
|(\phi^{t}_{F})^{*}[X_{R_{t}},X_{F}]|_{\frac{1}{4}\eta
s,D_{2}}&\leq& c|[X_{R_{t}},X_{F}]|_{\eta s,D_{4}}\leq
\eta^{-2}|[X_{R_{t}},X_{F}]|_{s,\tilde{D}} \nonumber\\
&\leq&
c_{6}(r-r_{+})^{-1}\eta^{-2}\gamma\mu^{2}\Gamma(r-r_+)C(a-a_+),\nonumber
\end{eqnarray}
and
\begin{eqnarray}
&~&|(\phi_{F}^{t})^{*}[X_{R_{t}},X_{F}]|^{\mathcal
{L}}_{\frac{1}{4}\eta
s,D_{2}}\nonumber\\&\leq&~c|[X_{R_{t}},X_{F}]|^{\mathcal {L}}_{\eta
s,D_{4}}+\frac{c}{(r-r_+)\eta^2}|[X_{R_{t}},X_{F}]|_{\eta
s,D_{4}}|X_F|^{{\mathcal
{L}}}_{s,\hat{D}}\nonumber\\
&\leq&~c\eta^{-2}|[X_{R_{t}},X_{F}]|^{\mathcal
{L}}_{s,\tilde{D}}+\frac{c}{(r-r_+)\eta^4}|[X_{R_{t}},X_{F}]|_{s,\tilde{D}}
|X_F|^{\mathcal {L}}_{s,\hat{D}}\nonumber\\
&\leq&~c_{6}(r-r_+)^{-1}\eta^{-2}M\mu^{2}\Gamma(r-r_+)C(a-a_+)\nonumber\\
&~&+c_{6}(r-r_+)^{-2}\eta^{-4}M\mu^{3}\Gamma^{2}(r-r_+)C^2(a-a_+).\nonumber
\end{eqnarray}
\begin{lemma}~Assume that
\begin{itemize}
\item[{\rm \bf
H8)}]$c_{6}
(\gamma\mu^{2}+(r-r_{+})^{-1}\eta^{-2}\gamma\mu^{2}\Gamma(r-r_+))C(a-a_+)\leq
\gamma_+\mu_+,$
\item[{\rm \bf
H9)}]$c_6(r-r_+)^{-1}\eta^{-2}M\mu^{2}\Gamma(r-r_+)C(a-a_+)+\\c_6
(r-r_+)^{-2}\eta^{-4}M\mu^{3}\Gamma^{2}(r-r_+)C^{2}(a-a_+)+c_{6}M\mu^{2}\leq
M_+\mu_+$.
\end{itemize}
Then we have
$$
|X_{P_+}|^{a_+,\bar{p}}_{s_+,D_+}\leq
\gamma_+\mu_+,~~~|X_{P_+}|^{{\mathcal
{L}}_{a_+,\bar{p}}}_{s_+,D_+}\leq M_+\mu_+.
$$
\end{lemma}
\begin{proof}
Note that
\begin{eqnarray}
|X_P|^{a_+,\bar{p}}_{s_+,D_{+}}&\leq&|X_{(P-R)\circ\phi_{F}^{1}}|^{a_+,\bar{p}}_{s_+,D_{+}}
+|\int_{0}^{1}(\phi^{t}_{F})^{*}[X_{R_{t}},X_{F}]dt|\nonumber\\
&\leq&|X_{(P-R)\circ\phi_{F}^{1}}|^{a_+,\bar{p}}_{\frac{1}{2}\eta
s,D_{4}}+
|(\phi^{t}_{F})^{*}[X_{R_{t}},X_{F}]|^{a_+,\bar{p}}_{\frac{1}{4}\eta s,D_{2}}\nonumber\\
&\leq&\gamma_+\mu_+,
\end{eqnarray}
and
\begin{eqnarray}
|X_{P_+}|^{{\mathcal
{L}}_{a_+,\bar{p}}}_{s_+,D_+}&\leq&|X_{(P-R)\circ\phi_{F}^{1}}|^{{\mathcal
{L}}_{a_+,\bar{p}}}_{s_+,D_+}+|(\phi_{F}^{t})^{*}[X_{R_{t}},X_{F}]|^{{\mathcal
{L}}_{a_+,\bar{p}}}_{\frac{1}{4}\eta s,D_{2}}\nonumber\\
 &\leq& M_+\mu_+.
\end{eqnarray}

For Theorem A, the proof is the same as
$C(a-a_+)=\sum_{0<i<I_+}i^{10}{\rm e}^{-i(a-a_+)}$.
\qquad\end{proof}

 In subsection 1.1, for Theorem A', we assume that the perturbation $P$ has
the special form. So to accomplish one KAM step, we need to prove
the new perturbation $P_+\in {\mathcal {A}}$.

\begin{lemma}~ If the function $G_1(y,x,z,\bar{z}),~G_2(y,x,z,\bar{z})\in
{\mathcal {A}}$, then $G_1\pm G_2,~\{G_1,G_2\}\in{\mathcal {A}}$.
\end{lemma}
\begin{proof}See \cite{geng} for details.\qquad\end{proof}

Note that
\begin{eqnarray*}
P_+&=&P-R+\frac{1}{2!}\{\{N,F\},F\}+\frac{1}{2!}\{\{P,F\}F\}+\cdots,\\
\{N,F\}&=&P_{0000}+\langle
P_{0100},y\rangle+\sum_{i}P^{011}_{i}z_i\bar{z}_i.
\end{eqnarray*}
It is easy to see $P,~R,~F,~\{N,F\}\in{\mathcal {A}}$, according to
the lemma above, we know $P_+\in{\mathcal {A}}$.

\section{Proof of Main Theorem}
\subsection{Iteration Lemma}
In this subsection, we will summarize {\rm section 3} and choose
iteration sequences.

 Set
\begin{eqnarray}
r_\nu&=&r_0(1-\displaystyle\sum_{i=1}^\nu\displaystyle\frac{1}{2^{i+1
}}),\nonumber\\
a_\nu&=&a_0(1-\displaystyle\sum_{i=1}^\nu\displaystyle\frac{1}{2^{i+1
}}),\nonumber\\
\gamma_\nu&=&\gamma_0(1-\displaystyle\sum_{i=1}^\nu\displaystyle\frac{1}{2^{i+1
}}),\nonumber\\
M_{\nu}&=&M_0(2-2^{\nu+1}),\nonumber\\
s_{\nu}&=&\eta_{\nu-1}
s_{\nu-1},~~\eta_{\nu}=\mu_{\nu}^{\frac{1}{3}},~~
\mu_{\nu}=s_{\nu}^{\frac{1}{2}}=\mu_{\nu-1}^{\frac{7}{6}},\nonumber\\
K_{\nu}&=&([\log\frac{1}{\mu_{\nu-1}}]+1)^{3\alpha_1},~~
I_{\nu}=([\log\frac{1}{\mu_{\nu-1}}]+1)^{3\alpha_2},\nonumber\\
D^{\nu-1}_{\frac a4,p;i}&=&D_{\frac
a4,p}(r_{\nu}+\frac{i-1}{8}(r_{\nu-1}-r_{\nu}),i\eta_{\nu-1}
s_{\nu-1}),~~
i=1,2,\cdots,8,\nonumber\\
\hat D^{\nu}_{\frac a4,p}(\beta)&=&D_{\frac a4,p}(r_{\nu}+\frac{7}{8}(r_{\nu-1}-r_{\nu}),\beta),~~~\beta>0,\nonumber\\
\tilde{D}^{\nu}_{\frac a4,p}(\beta)&=&D_{\frac
a4,p}(r_{\nu}+\frac{3}{4}(r_{\nu-1}-r_\nu),\frac{\beta}{2}),~~~\beta>0,\nonumber
\end{eqnarray}
for all $\nu=1,2,\cdots$.

\begin{lemma}~
 If $\mu_0=\mu_0(a,p,\bar{p},r,s,m,\tau)$, or equivalently,
$\mu_*=\mu_*(a,p,\bar{p},r,s,m,\tau)$, is sufficiently small, then
the KAM step described in {\rm Section~3} is valid for all
$\nu=0,1,\cdots$. Consider the  sequences:
 $$
 \Lambda_{\nu},~ H_{\nu},~
  N_{\nu},~ e_{\nu},~\omega_{\nu},
 ~M_{\nu},~P_{\nu},~\Phi_{\nu},
 $$
$\nu=1,2,\cdots$. The following properties hold:
\begin{itemize}
\item[{\rm 1)}] $\Phi_{\nu}:D_{\frac a4,p,3}\times \Pi_{\nu+1}\longrightarrow
D_{\frac a4,p}\times\Pi_{\nu}$ is symplectic, and is of class
$C^{2}$ smooth and depends Lipshcitz-continuously on parameter,
\begin{eqnarray}
|D^{\imath}(\Phi_{\nu+1}-{\rm id})|^{a_{\nu+1},p}_{D_{\frac
a4,p;3}\times\Pi_{\nu+1}}
 \le \frac{\mu_{*}^{\frac14}}{2^{\nu}},\nonumber
 \end{eqnarray}
 where $\imath=0,1$~$\mu_{*}=\mu_{0}^{1-\sigma},~\sigma\in[\frac{3}{4},1)$.
\item[{\rm 2)}] On ${\hat D}_{\frac a4,p}\times \Pi_{\nu+1}$,
$$
H_{\nu+1}=H_{\nu}\circ\Phi_{\nu+1}=N_{\nu+1}+P_{\nu+1},
$$
where
\begin{eqnarray}
 &&|\omega_{\nu+1}
-\omega_{\nu}|_{\Pi_{\nu+1}}\le
\gamma_0\frac{\mu_{*}}{2^{\nu+1}}, \nonumber\\
 &&|\omega_{\nu+1}-\omega_{0}|_{\Pi_{\nu+1}}
 \le \gamma_0\mu_{*},\nonumber\\
&&|\Omega^{\nu+1} -\Omega^{\nu}|_{-\delta,\Pi_{\nu+1}}\le
\gamma_0\frac{\mu_{*}}{2^{\nu+1}}, \nonumber\\
 &&|\Omega^{\nu+1} -\Omega^{0}|_{-\delta,\Pi_{\nu+1}}
 \le \gamma_0\mu_{*},\nonumber\\
 &&|\omega_{\nu+1}
-\omega_{\nu}|^{\mathcal {L}}_{\Pi_{\nu+1}}\le
M_0\frac{\mu_{*}}{2^{\nu+1}}, \nonumber\\
&&|\omega_{\nu+1} -\omega_{\nu}|^{\mathcal {L}}_{\Pi_{\nu+1}} \le
M_0\mu_{*},\nonumber\\
 &&|\Omega^{\nu+1}
-\Omega^{\nu}|^{{\mathcal {L}}}_{-\delta,\Pi_{\nu+1}}\le
M_0\frac{\mu_{*}}{2^{\nu+1}},\nonumber\\
&&|\Omega^{\nu+1} -\Omega^{0}|^{{\mathcal
{L}}}_{-\delta,\Pi_{\nu+1}}
 \le M_0 \mu_{*},\nonumber\\
 &&|\omega_{\nu+1}|^{\mathcal
{L}}_{\Pi_{\nu+1}}+|\Omega^{\nu+1}|_{-\delta,\Pi_{\nu+1}}^{{\mathcal
{L}}}\leq M_{\nu+1},\nonumber\\
 &&|X_{P_{\nu+1}}|_{s_{\nu+1},D^{\nu+1}}^{a_{\nu+1,\bar{p}}}\leq \gamma_{\nu+1}
 \mu_{\nu+1},\nonumber\\
&&|X_{P_{\nu+1}}|^{{\mathcal
{L}}_{a_{\nu+1},\bar{p}}}_{s_{\nu+1},D^{\nu+1}}\leq
M_{\nu+1}\mu_{\nu+1}.\nonumber
\end{eqnarray}
\item[{\rm 3)}]
$$
\Pi_{\nu+1}=\Pi_{\nu}\setminus\bigcup_{K_{\nu}<|k|\leq
K_{\nu+1},\\~0<|l|\leq2}R^{\nu+1}_{kl},
$$
where
$$
\tilde{R}^{\nu+1}=\bigcup_{K_{\nu}<|k|\leq
K_{\nu+1},\\~0<|l|\leq2}R^{\nu+1}_{kl}=\{\xi\in\Pi_{\nu}: |\langle
k,\omega_{\nu}(\xi)\rangle+\langle l,\Omega_{\nu}(\xi)\rangle|\leq
\frac{\gamma_{\nu}}{A_{k,l}}\}.
$$
\end{itemize}
\end{lemma}
\begin{proof} We first verify H1)-H9) for all $\nu=0,1\cdots$ to guarantee
the KAM cycle in {\rm section 3}. Generally, we choose $\mu_{0}$
sufficiently small so that H1)-H9) hold for $\nu=1$ and let
$r_{0},~\gamma_{0}=1$ for simplicity. Note that
\begin{equation}
\mu_{\nu}=\mu_{0}^{(7/6)^\nu},~~s_{\nu}=s_{0}^{(7/6)^\nu},
~~r_{\nu}-r_{\nu+1}=\frac{1}{2^{\nu+2}},~~a_{\nu}-a_{\nu+1}=\frac{1}{2^{\nu+2}}.\label{mu,s,r}
\end{equation}
Let $c_0=\max\{c_1,\cdots,c_6\}$.

It follows from(\ref{mu,s,r}) that
\begin{eqnarray*}
&~&\log(n+4)!+(\nu+6)(n+4)\log2+3\alpha_1
n\log([\log\frac{1}{\mu_{\nu}}]+1)\\
&~&-\frac{K_{\nu+1}}{2^{\nu+2}}-\log\mu_{\nu}\nonumber\\
&\leq &\log(n+4)!+(\nu+6)(n+4)\log2+3\alpha_1 n
\log(\log\frac{1}{\mu_{\nu}}+2)\nonumber\\
&~&(7/6)^{\nu}\log\frac{1}{\mu_0}-\frac{(7/6)^{3\alpha_1\nu}}{2^{\nu+2}}\log\frac{1}{\mu_0}\\
&\leq& 0.\nonumber
\end{eqnarray*}
Choose $\mu_{0}$ and $\alpha_1$ such that $(7/6)^{3\alpha_1-1}\geq
8$, so
\begin{eqnarray*}
\displaystyle\int_{K_{\nu+1}}^\infty t^{n+3} {\rm
e}^{-\frac{t}{2^{\nu+3}}}{\rm d}t\leq
(n+4)!2^{(\nu+6)(n+3)}K_{\nu+1}^n{\rm
e}^{-\frac{K_{\nu+1}}{2^{\nu+2}}}\leq \mu_{\nu},
\end{eqnarray*}
i.e.,  H1) holds.  Similarly, choose $\alpha_2$ such that
\begin{eqnarray*}
\displaystyle\int_{I_{\nu+1}}^\infty t^{\rho+4} {\rm
e}^{-\frac{t}{2^{\nu+3}}}{\rm d}t\leq
(\rho+5)!2^{(\rho+4)(\nu+2)}I_{\nu+1}{\rm
e}^{-\frac{K_{\nu+1}}{2^{\nu+2}}}\leq \mu_{\nu},
\end{eqnarray*}
so H2) holds.

Then we consider H7),
\begin{eqnarray}
&~&c_{0}\mu_{\nu}(K_{\nu+1}+K_{\nu+1}^{\tau+1})I_{\nu+1}^{c(\rho)}\\
&\le& 2c_{0}\mu_{\nu}(\log\frac{1}{\mu_\nu}+2)^{3\alpha_1(\tau+1)}
(\log\frac{1}{\mu_\nu}+2)^{3\alpha_2c(\rho)}\nonumber\\
&\leq&
2c_{0}(2\log\frac{1}{\mu_0})^{3\alpha_1(\tau+1)+3\alpha_2c(\rho)}
\mu_{0}^{(\frac{7}{6})^{\nu}}(7/6)^{3\alpha_1(\tau+1)+3\alpha_2c(\rho)\nu}\nonumber\\
&\leq&2c_{0}(2\log\frac{1}{\mu_0})^{3\alpha_1(\tau+1)+3\alpha_2c(\rho)}
\mu_{0}(\mu_{0}^{\frac{1}{6}}(7/6)^{3\alpha_1(\tau+1)+3\alpha_2c(\rho)})^{\nu},\nonumber
\end{eqnarray}
since $\mu_{0}^{(\frac{7}{6})^{\nu}}\leq\mu_{0}^{1+\frac{\nu}{6}}$.
Choose $\mu_0$ such that
$$
\mu_{0}^{\frac{1}{6}}(7/6)^{3\alpha_1(\tau+1)+3\alpha_2c(\rho)}\leq
\frac{1}{3}.
$$
So H7) holds. Also
$$
4c_{0}\mu_{0}^{1+\frac{\nu}{6}}<\frac{1}{2^{\nu+2}}
$$
implies H6).

One can find that, in the proof above, choosing
$\rho=1,~c(\rho)=\frac52$, H1),~H2),~H6),~H7)(or H7)') also hold.

 Note that
\begin{eqnarray}
\Gamma_{\nu}&=&\Gamma(r_\nu-r_{\nu+1})\le
\displaystyle\sum_{0<|k|\leq
K_{\nu+1}}|k|^{4\tau+4}{\rm e}^{-|k|\frac{1}{2^{\nu+6}}}\nonumber\\
&\leq& \int_{1}^{\infty}\lambda^{4\tau+n+4}
{\rm e}^{-\frac{\lambda}{2^{\nu+6}}}{\rm d}\lambda\nonumber\\
&\leq&(4[\tau]+n+4)!2^{(\nu+6)(4\tau+n+4)},\nonumber
\end{eqnarray}
and
\begin{eqnarray*}
C_{\nu}&=&C(a_{\nu}-a_{\nu+1})=\sum_{0<|i|\le
I_{\nu+1}}|i|^{c(\rho)+10}{\rm
e}^{-(a_{\nu}-a_{\nu+1})|i|}\\
&\le&\int_{1}^{\infty}\lambda^{c(\rho)+10}{\rm
e}^{\frac{\lambda}{2^{\nu+2}}}{\rm d}\lambda\\
&\le&(c(\rho+10))!2^{(\nu+2)(c(\rho)+10)}.
\end{eqnarray*}

Let
\begin{eqnarray*}
a^{*}&=&(4[\tau]+n+4)!(c(\rho)+10)!2^{6(4\tau+n+6)+2(c(\rho)+10)},\\
b^{*}&=&4\tau+n+c(\rho)+14.
\end{eqnarray*}
Then
\begin{equation}
\Gamma_{\nu}C_{\nu}\le a^{*}2^{b^{*}\nu}.
\end{equation}
It is clear that, H3)-H5)and H8)-H9) are equivalent to
\begin{eqnarray}
32c_{0}a^{*}\mu_\nu2^{(b^*+1)\nu}&\leq&1,\label{h2}\\
\frac{1}{5}c_{0}a^{*}\mu_\nu^{\frac13}2^{b^{*}\nu}&\leq&1.\nonumber
\end{eqnarray}
Observing H8), H9) yields
\begin{eqnarray}\label{H7}
&~&c_0\frac{\gamma_\nu\mu^2_\nu}{\gamma_{\nu+1}\mu_{\nu+1}}+
c_0\frac{\gamma_\nu\mu^2_\nu\Gamma_\nu C_{\nu}}{(r_{\nu}-r_{\nu+1})\eta_{\nu}^{2}\gamma_{\nu+1}\mu_{\nu+1}}\nonumber\\
&\leq&2c_{0}\mu_{\nu}^{\frac56}+8c_{0}2^{\nu}\mu_{\nu}^{\frac16}\Gamma_{\nu}C_{\nu}\\
&\leq&~1,\nonumber
\end{eqnarray}
and
\begin{eqnarray}\label{H8}
&~&c_0\frac{M_\nu\mu^2_\nu}{M_{\nu+1}\mu_{\nu+1}}+
c_0\frac{M_\nu\mu^2_\nu\Gamma_\nu
C_{\nu}}{(r_{\nu}-r_{\nu+1})\eta_{\nu}^{2} M_{\nu+1}\mu_{\nu+1}}\nonumber\\
&~&+ c_0\frac{M_\nu\mu^3_\nu\Gamma^{2}_\nu
C^2_{\nu}}{(r_{\nu}-r_{\nu+1})^{2}\eta_{\nu}^{4}
M_{\nu+1}\mu_{\nu+1}}\nonumber\\
&\leq&\frac{c_{0}}{2}\mu_0^{\frac56}+4c_{0}2^{\nu}\mu_{\nu}^{\frac16}\Gamma_{\nu}C_{\nu}
+4c_{0}2^{2\nu}\mu_{\nu}^{\frac12}\Gamma^{2}_{\nu}C^2_{\nu}\\
&\leq&~1.\nonumber
\end{eqnarray}
Combining the above, we only prove
\begin{equation}\label{mugamma}
8c_{0}2^{\nu}\mu_{\nu}^{\frac16}\Gamma_{\nu}\leq \frac{1}{3}.
\end{equation} In fact
$$
8c_{0}a^*\mu_{\nu}^{\frac16}2^{(b^{*}+1)\nu}\leq
8c_{0}a^*(\mu_{0}^{\frac{1}{6}})^{(1+(\nu/6))}2^{(b^{*}+1)\nu}
\leq8c_{0}a^*\mu_{0}^{\frac{1}{3}}(\mu_{0}^{\frac{1}{36}}2^{b^{*}+1})^{\nu}.
$$
Choose $\mu_{0}$ sufficiently small such that (\rm\ref{mugamma}) is
verified. And by (\rm\ref{h2})-(\rm\ref{H8}), we know that H3)- H5),
H8), H9) hold.

 From
(\ref{mugamma}), we have
$$
c_0\mu_{\nu}\Gamma_{\nu}\leq \frac{1}{16}\mu_{\nu}^{(5/6)} \leq
\frac{\mu_{*}}{2^{\nu}},
$$
where $\mu_{*}=\mu_{0}^{1-\sigma},~~\sigma\ge \frac{1}{2}$, so
\begin{eqnarray}
c_{0}\mu_{\nu}\Gamma_{\nu}\leq \frac{\mu_{*}}{2^{\nu}},~~~
c_{0}\frac{ M_\nu}{\gamma_{\nu}}\mu_{\nu}\Gamma_{\nu}\leq \frac{
M_\nu}{\gamma_{\nu}}\frac{\gamma^{a}\mu_{*}}{2^{\nu+1}},
\end{eqnarray}
for all $\nu=0,1\cdots$.

We now proceed the iterative schemes. First, it is clear to see the
part 1) and 2) of the lemma hold for $\nu=1$. Then assume that for
some $\nu_*$, 1) and 2) are true for all $\nu=1,\cdots,\nu_*$. Then
by Lemmas in section 2 and arguments above, we claim that the KAM
iteration is valid for $\nu=\nu_*+1$.

For the proof of Theorem A, we choose $c(\rho)=\frac52$, and
\begin{eqnarray*}
C_{\nu}&\le&\int_{1}^{\infty}\lambda^{6}{\rm
e}^{\frac{\lambda}{2^{\nu+2}}}{\rm d}\lambda\le6!2^{6\nu+12},\\
a^{*}&=&(4[\tau]+n+4)!6!2^{6(4\tau+n+2)},\\
b^{*}&=&4\tau+n+12.
\end{eqnarray*}
The other  proof is the same.

 It is easy to see that 3) holds for $\nu=1$. For $\nu>1$, from lemma 2.7 we have
\begin{eqnarray}
\Pi_{\nu}&=&\{ \xi\in \Pi_{\nu}: |\langle
k,\omega_{\nu}(\xi)\rangle+\langle
l,\Omega_{\nu}(\xi)\rangle|>\frac{\gamma_{\nu}}{A_{k,l}},\nonumber\\
&~&~~~\hbox{\rm
  for all}\,\,0<|k|\le K_{\nu},~0<|l|\leq2\}.\nonumber
\end{eqnarray}
Set
\begin{eqnarray}
\tilde{R}_{\nu+1}&=&\{\xi\in\Pi_{\nu}: |\langle
k,\omega_{\nu}(\xi)\rangle+\langle l,\Omega_{\nu}(\xi)\rangle|\leq
\frac{\gamma_{\nu}}{A_{k,l}}, \nonumber\\
&~&~~K_{\nu}\leq|k|\leq K_{\nu+1},~0<|l|\leq2\}.\nonumber
\end{eqnarray}
Then
\begin{eqnarray}
\Pi_{\nu+1}&=&\{ \xi\in \Pi: |\langle
k,\omega_\nu(\xi)\rangle+\langle
l,\Omega_\nu(\xi)\rangle|>\frac{\gamma_{\nu}}{A_{k,l}},\nonumber\\
&~&~~~\hbox{\rm
  for all}\,\,0<|k|\le K_{\nu+1}\,~0<|l|\leq2\}\nonumber\\
&=&\Pi_{\nu}\setminus\tilde{R}_{\nu+1}.\nonumber
\end{eqnarray}
The lemma is now complete.\qquad\end{proof}

\subsection{Convergence}
In this section,  we prove the convergence of the sequences from
{\rm subsection~3.1}. Let
\begin{eqnarray}
& & \Psi^\nu=   \Phi_0\circ\Phi_1\circ\cdots\circ\Phi_\nu:
D_{\frac a4,p}\times\Pi_{\nu+1}\rightarrow  D_0,\nonumber\\
& & H\circ\Psi^\nu=H_{\nu}=N_{\nu}+P_{\nu},\nonumber
\end{eqnarray}
for $\nu=0,1,\cdots$, which satisfy all properties described in
Lemma~3.1. By the iteration lemma, we claim that
 $\Psi^\nu$ is convergent to a function $\Psi^\infty\in C^{1}(D(\frac
{r_0}2,0)\times \Pi_{*})$, in $C^{2}(D(\frac {r_0}2,0)\times
\Pi_{*})$, where
$$
\Pi_{*}=\cap_{\nu=0}^\infty \Pi_\nu.
$$
Then $\Pi_{*}$ is a Cantor set, and $\{\Psi_{\xi}:\xi\in \Pi_{*}\}$
is a Lipschitz continuous family of real analytic, symplectic
transformations on $D(\frac {r_0}2,0)$.

By part 2) of lemma~3.1, it is easy to see that $\omega_{\nu}$,
$\Omega_{\nu}$ converge uniformly on $\Pi^{*}$. We denote
$\omega_*$, $\Omega^*$ as their limits, respectively. Then, we have
 \begin{eqnarray}
|\omega_*-\omega_0|_{\Pi_*}&=&O(\gamma_0\mu_*),\nonumber\\
 |\Omega^*-\Omega^0|_{-\delta,\Pi_*}&=&O(\gamma_0\mu_*),\nonumber\\
 |\omega_*-\omega_0|_{\Pi_*}^{\mathcal {L}}&=&O(M_0\mu_*),\nonumber\\
 |\Omega^*- \Omega^0|_{-\delta,\Pi_*}^{{\mathcal {L}}}
 &=&O(M_0\mu_*).\nonumber
\end{eqnarray}
Thus, on $G_*$, $|N_\nu-N_{\infty}|$ converges uniformly  to 0,
and
$$
P_\nu=H\circ\Psi^\nu-N_\nu
$$
converges uniformly to
$$
P_\infty=H\circ\Psi^\infty-N_\infty.
$$
 Clearly, these limits above are  uniformly Lipshchitz continuous
in $\xi\in \Pi^{*}$ and real analytic in $(x,y,z,\bar{z})\in
D_{\frac{a}{2},p}(\frac {r_0}2,\frac {s _0}2)$.

Note that
\begin{eqnarray}
|X_{P_\nu}|^{a_{\nu}, \bar{p}}& \leq&  \gamma_\nu
\mu_\nu,\nonumber\\
 |X_{P_\nu}|^{{\mathcal {L}}_{a_{\nu},\bar{p}}}&\leq& M_{\nu}\mu_{\nu}.\nonumber
\end{eqnarray}
It means that for any $\xi\in\Pi_{*}$,~$m\in {\bf
Z}^{n}_+,~q,\bar{q}\in Z_{1}^{\rho}$ with $2|j|+|q+\bar{q}|\leq 2$,
\begin{equation}
|\partial_y^{m} \partial_z^q\partial_{\bar{z}}^{\bar{q}}
P_\nu|_{D_{\nu}}\leq\gamma_\nu \mu_\nu.
\end{equation}
 Since the right hand side of the above
converges to $0$ as $\nu\to \infty$, we have  that
$$
\partial_y^{j} \partial_z^q\partial_{\bar{z}}^{\bar{q}} P_\infty|_{(y,z,\bar{z})=0}= 0
$$
for all $x\in {\bf T}^n$, $\xi\in\Pi_{*}$, $m\in {\bf Z}^n_+$,
$q,~\bar{q}\in Z_1^{\rho}$ with $2|j|+|q+\bar{q}|\le 2$.

\section{Measure Estimates}
It remains to show the measure estimate. We induct a lemma first.
Recall that
\begin{eqnarray}
|\omega_{\nu}|^{\mathcal
{L}}+|\Omega^{\nu}|_{-\delta,\Pi_{\nu}}^{\mathcal
{L}}&\leq& M_{\nu}\leq 2M_{0},\\
|\omega_{\nu}-\omega_{0}|,
~|\Omega^{\nu}-\Omega^{0}|_{-\delta,\Pi_{\nu}}
&\leq&\gamma_{0}~\mu_{*},\\
R^{\nu+1}_{kl}=\{\xi\in\Pi_{\nu}: |\langle
k,\omega_{\nu}(\xi)\rangle+\langle l,\Omega_{\nu}(\xi)\rangle|&\leq&
\frac{\gamma_{\nu}}{A_{k,l}}\}.
\end{eqnarray}

\begin{lemma}~Define \begin{eqnarray}
 R_{kl}(\alpha_{k,ij})&:=&\{\xi\in\Pi_{\nu}: |\langle
k,\omega_{\nu}(\xi)\rangle+\langle l,\Omega_{\nu}(\xi)\rangle|\leq
\alpha_{k,ij}\},\nonumber
\end{eqnarray}
where $k\in {\bf Z}^{n}\setminus \{0\},$~$i,j$ denote the sites of
non-zero components of $l$, and $\alpha_{k,ij}$ is a positive
constant depending on $k,i,j$.

If $|\Omega^{\nu}|_{-\delta,\Pi_{\nu}}^{\mathcal {L}}\leq M$ and
$|k|\ge 16M$, then for fixed $l\in\Lambda$, we have
\begin{equation}\label{5.4}
|R_{kl}(\alpha_{k,ij})|\prec c(n)\frac{\alpha_{k,ij}}{|k|},
\end{equation}
where $\prec$ means $\cdot\leq c \cdot$, $c$ is a positive constant
independent of any parameters, and $c(n)>0$ depends only on $n$.
\end{lemma}
\begin{proof} Fix $\varpi\in\{-1,1\}^{n}$ such that $|k|=k\cdot\varpi$, and
write $\omega_{\nu}=a\varpi+\rho,~\varpi\perp\rho$. Let $a$ be a
variable. Then for $t>s$,
$$
\langle k,\omega_{\nu}\rangle=\langle
k,\omega_{0}\rangle|^t_s-|\langle k,\omega_\nu-\omega_0\rangle|
\ge|k|(t-s)-\frac 12|k|(t-s)=\frac 12|k|(t-s),
$$
and
$$
|\langle l,
\Omega^{\nu}\rangle|\leq|l|_{\delta}|\Omega^{\nu}|_{-\delta,\Pi_{\nu}}^{\mathcal
{L}}(t-s)\leq 2M|l|_{\delta}(t-s)\leq \frac 14|k|(t-s),
$$
since $|l|_{-\delta}\le 2$. Hence,
$$
\langle k, \omega_{\nu}\rangle+\langle l,\Omega^{\nu} \rangle
\geq\frac 14|k|(t-s).
$$
This can deduce (\rm\ref{5.4}). The details can be seen in lemma 5
of \cite{posel1} or lemma 5.6 of \cite{posel2}. \qquad\end{proof}

\begin{lemma}~ There exists a positive constant $c_{7}$ such that
\begin{eqnarray*}
|\langle l,\Omega^{\nu}\rangle|\ge\left\{\begin{array}{ll}
c_7\langle l\rangle_{d},~~~&l\in\Lambda_+;\\
|c_7\langle
l\rangle_{d}-2\gamma_{0}\mu_*|,&l\in\Lambda_-,\end{array}\right.
\end{eqnarray*}
for all $\xi\in\Pi_\nu$.
\end{lemma}
\begin{proof}~ Consider the case $l\in\Lambda$. By A1) and A2) we have
$$
\frac{|\langle l,\Omega^{0}\rangle|}{\langle
l\rangle_{d}}\rightarrow1,
$$
uniformly in $\xi\in\Pi_0$, i.e., there is a small positive
$\varepsilon$,
$$
|\langle l,\Omega^{0}\rangle|\ge (1-\varepsilon)\langle
l\rangle_{d}.
$$
For $l\in\Lambda_+$, since $|l|_{\delta}\leq \langle l\rangle_{d}$,
$$
|\langle l,\Omega^{\nu}-\Omega^{0}\rangle|\le
|l|_{\delta}|\Omega^{\nu}-\Omega^{0}|_{-\delta,\Pi_{\nu}}\le\langle
l\rangle_{d}\gamma_0\mu_*,
$$
and
\begin{eqnarray}
|\langle l,\Omega^{\nu}\rangle|&\ge&|\langle
l,\Omega^{0}\rangle|-|\langle l,\Omega^{\nu}-\Omega^0\rangle|\nonumber\\
&\ge&(1-\varepsilon-\gamma_{0}\mu_*)\langle l\rangle_{d}.\nonumber
\end{eqnarray}
Choosing $\varepsilon$ such that
$1-\varepsilon-\gamma_{0}\mu_*>c_7$, we prove the first part of the
results.

For $l\in\Lambda_-$, $|l|_{\delta}\le 2$, we have
$$
|\langle l,\Omega^{\nu}-\Omega^{0}\rangle|\le
|l|_{\delta}|\Omega^{\nu}-\Omega^{0}|_{-\delta,\Pi_{\nu}}\le2\gamma_0\mu_*,
$$
and
\begin{eqnarray}
|\langle l,\Omega^{\nu}\rangle|&\ge&||\langle
l,\Omega^{0}\rangle|-|\langle l,\Omega^{\nu}-\Omega^0\rangle||\nonumber\\
&\ge&|(1-\epsilon)\langle l\rangle_{d}-2\gamma_0\mu_*|.\nonumber
\end{eqnarray}
\qquad\end{proof}

\begin{lemma} If $R^{\nu+1}_{kl}\ne\emptyset$, and $\gamma_{0}\leq \frac
{c_7}{2}$, then there exits a positive constance $c_8$, such that
\begin{eqnarray*}
\langle l\rangle_{d}\le c_8|k|,
\end{eqnarray*}
for $l\in\Lambda$.
\end{lemma}

\begin{proof}~ If $R^{\nu+1}_{kl}$ is not empty, then for $l\in\Lambda_+$ or
$l\in\Lambda_-$, there exits some $\xi\in\Pi_{\nu}$, such that
$|\langle k,\omega_{\nu} \rangle+\langle
l,\Omega^{\nu}\rangle|<\gamma_0\langle l\rangle_{d}$ or $|\langle
k,\omega_{\nu} \rangle+\langle l,\Omega^{\nu}\rangle|<\gamma_0$,
respectively.

Consider $l\in\Lambda_{+}$,
\begin{eqnarray*}
|k||\omega_{\nu}|&\ge&|\langle
l,\Omega^{\nu}\rangle|-\gamma_{0}\langle l\rangle_{d}\nonumber\\
&\ge& \frac{c_7}{2}\langle l\rangle_{d},\nonumber
\end{eqnarray*}
i.e.,
$$
\langle l\rangle_{d}\le\frac{4}{c_7}|k||\omega_0|_{\Pi},
$$
with $c_8\ge\frac{4}{c_7}|\omega_0|_{\Pi}$.

For $l\in\Lambda_-$, we have
\begin{eqnarray*}
|k||\omega_{\nu}|+\gamma_0&\ge&|\langle l,\Omega^{\nu}\rangle|\ge
|c_7\langle l\rangle_{d}-2\gamma_{0}\mu_*|,
\end{eqnarray*}
that is
\begin{eqnarray*}
\langle l\rangle_{d}&\le& \frac{1}{c_7}
(\gamma_0(2\mu_*+1)+2|k||\omega_0|_{\Pi})\le \frac{2}{c_7}(\gamma_0+|k||\omega_0|_{\Pi}),\nonumber\\
&\le&\frac{2}{c_7}(\gamma_0|k|+|k||\omega_0|_{\Pi}),\nonumber\\
&\le&\frac{2}{c_7}(\gamma_0+|\omega_0|_{\Pi})|k|,\nonumber
\end{eqnarray*}
with $c_8\ge\frac{2}{c_7}(\gamma_0+|\omega_0|_{\Pi})$.
\qquad\end{proof}

Then we decompose $\tilde{R}_{\nu+1}$ into two parts
\begin{eqnarray}
\tilde{R}_{\nu+1}&\equiv&\cup_{K_{\nu}<|k|\leq
K_{\nu+1}}(\cup_{l\in\Lambda_{-}}R_{kl}^{\nu+1}+\cup_{l\in\Lambda_{+}}R_{kl}^{\nu+1})\nonumber\\
&=& ~~~~~~\cup_{K_{\nu}<|k|\leq
K_{\nu+1}}I~~~~~+~~~~~\cup_{K_{\nu}<|k|\leq
K_{\nu+1}}II.~~~~~~~~~\nonumber
\end{eqnarray}

\begin{itemize}
\item Case I:~~$l\in\Lambda_{-}$.
From lemma 4.1, we have
$$
|R_{kl}^{\nu+1}|\prec c(n)
\frac{\gamma_{\nu}}{|k|(1+|k|^{\tau+1})i^{c(\rho)}}\prec c(n)
\frac{\gamma_{0}}{|k|(1+|k|^{\tau+1})i^{c(\rho)}}
$$
for $|k|\ge 16M,~l\in\Lambda_-$.

\begin{lemma}
~ The following holds for $|k|$ sufficiently large,
$$|I|\prec c(n) \frac{\gamma_{0}|k|^{\frac{c_1(\rho)}{d}}}{1+|k|^{\tau}},$$
where $c_1(\rho)$ is a constant depending only on $\rho$.
\end{lemma}
\begin{proof} ~Since $l\in\Lambda_-$, $\langle l\rangle_d=||i|^d-|j|^d|$, by
lemma 4.3, we have
$$
||i|^d-|j|^d|\leq c_8|k|,~~~|k|\ge 16M.
$$
For any fixed $i\in Z_1^{\rho}$, we have
\begin{eqnarray}
&~&{\rm card}\{j:||i|^d-|j|^d|\le c_8|k|\}\nonumber\\
&\le& {\rm card}\{j:~|j|\le |i|\}+{\rm card}\{j:~|j|^d\le c_8|k|+|i|^d\}\nonumber\\
&\le& i+{\rm card}\{j:~j^d\le (c_8+1)|k|i^d\}\le
|i|^{c_1(\rho)}+c|k|^{\frac{c(\rho)}d}|i|^{c_1(\rho)}\nonumber\\
&\le&(c+1)|k|^{\frac{c_1(\rho)}d}|i|^{c_1(\rho)}.\nonumber
\end{eqnarray}
Then for the fixed $i$,
$$
|R^{\nu+1}_{kl(i)}|\le |R^{\nu+1}_{kl}|{\rm card\{j:||i|^d-|j|^d|\le
c_8|k|\}}\le c(n)\frac{\gamma_0 |k|^{\frac
{c_1(\rho)}d}|i|^{c_1(\rho)}}{(1+|k|^{\tau})|i|^{c(\rho)}},
$$
and
\begin{eqnarray}
|I|&\leq&\sum_{i\in Z_1^{\rho}}^{\infty}|R_{kl(i)}^{\nu+1}|\prec
c(n)\frac{\gamma_{0}|k|^{\frac
{c_1(\rho)}d}}{(1+|k|^{\tau})}\sum_{i\in
Z_1^{\rho}}\frac{1}{|i|^{c(\rho)-c_1(\rho)}}.\nonumber
\end{eqnarray}
Choose $c(\rho)\ge c_1(\rho)+\rho$ such that $\sum_{i\in
Z_1^{\rho}}\frac{1}{i^{c(\rho)-c_1(\rho)}}$ converges to a constant.

\item Case I:~~$l\in\Lambda_{+}$.
By lemma 4.1, we have
\begin{eqnarray}
|R^{\nu+1}_{kl}|\prec c(n) \frac {\gamma_{\nu}}{1+|k|^{\tau}}\prec
c(n) \frac {\gamma_{0}}{(1+|k|^{\tau})},
\end{eqnarray}
for $|k|\geq 16M,~l\in\Lambda_+$.\qquad\end{proof}

\begin{lemma}~ The following holds
$$|II|\prec c(n) \gamma_{0}\frac{|k|^{\frac{2+c_1(\rho)}{d}}}{1+|k|^{\tau}}$$
for $|k|$ sufficiently large.
\end{lemma}
\begin{proof} Since $l\in\Lambda_+$, we have
\begin{eqnarray*}
\langle l\rangle_{d}=|i|^{d}~\rm or~ |j|^{d}~\rm or~else.
\end{eqnarray*}

\noindent{Then}
\begin{eqnarray}
&~&\rm card\{l:\langle l\rangle_{d}\le c^{-\frac{1}{d}}_{8}|k|\}\nonumber\\
&\le&\rm card\{l:|i|\le {c^{-\frac{1}{d}}_{8}|k|}^{\frac1d}\}+\rm
card\{l:|j|\le c^{-\frac{1}{d}}_{8}|k|^{\frac1d}\}\nonumber\\
&\le& 3c^{-\frac{2}{d}}_{8}|k|^{\frac{2+c_1(\rho)}{d}},\nonumber
\end{eqnarray}
\noindent{hence}
$$
|II|\prec
c(n)\gamma_0\frac{|k|^{\frac{2+c_1(\rho)}{d}}}{1+|k|^{\tau}}.
$$
\qquad\end{proof}
\end{itemize}
Combining the two cases, we have
$$
|\cup_{K_{\nu}<|k|<K_{\nu+1}}(I+II)|\le
c\gamma_0\sum_{K_{\nu}<|k|<K_{\nu+1}}\frac{|k|^{\frac{2+c_1(\rho)}{d}}}{1+|k|^{\tau}}.
$$

From lemma 3.1 3), we know that
$$
\Pi\setminus\Pi_{*}\subset\cup_{\nu=0}^{\infty}\tilde{R}_{\nu+1}.
$$
Then there exits a $\nu_{*}$ such that $K_{\nu_*}$ sufficiently
large and by choosing $\tau$, we have
\begin{eqnarray}
|\Pi\setminus\Pi_{*}|&\leq&\sum_{0}^{\nu_{*}}\tilde{R}_{\nu+1}
+\sum_{\nu_{*}}^{\infty}\sum_{_{K_{\nu}<|k|<K_{\nu+1}}}\tilde{R}_{\nu+1}\nonumber\\
&\le& O(\gamma_{0}),\nonumber
\end{eqnarray}
i.e. $|\Pi\setminus\Pi_{*}|\rightarrow 0$, as $\gamma_0\rightarrow
0$.

For the proof of Theorem A, we only show the lemmas and the key
point.

\begin{lemma} There exists a positive constant $c_{9}$ such that
\begin{eqnarray*}
|\langle l,\Omega^{\nu}\rangle|\ge\left\{\begin{array}{ll}
c_9\langle l\rangle_{d},~~~&l\in\Lambda_+;\\
|c_9\langle
l\rangle_{d}-2\gamma_{0}\mu_*|,&l\in\Lambda_-,\end{array}\right.
\end{eqnarray*}
for all $\xi\in\Pi_\nu$.
\end{lemma}
\begin{lemma}~If $R^{\nu+1}_{kl}\ne\emptyset$, and $\gamma_{0}\leq \frac
{c_9}{2}$, then there exits a positive constance $c_{10}$, such that
\begin{eqnarray*}
\langle l\rangle_{d}\le c_{10}|k|,
\end{eqnarray*}
for $l\in\Lambda$.
\end{lemma}
Consider the case ~~$l\in\Lambda_{-}$. From lemma 4.1, we have
$$
|R_{kl}^{\nu+1}|\prec c(n)
\frac{\gamma_{\nu}}{|k|(1+|k|^{\tau+1})i^{\frac52}}\prec c(n)
\frac{\gamma_{0}}{|k|(1+|k|^{\tau+1})i^{\frac52}}
$$
for $|k|\ge 16M,~l\in\Lambda_-$.

\begin{lemma}~The following holds
$$|I|\prec c(n) \frac{\gamma_{0}}{1+|k|^{\tau}}$$
for fixed $k\ge 16M$.
\end{lemma}
\begin{proof} ~Since $l\in\Lambda_-$, $\langle l\rangle_d=|i^d-j^d|$, by
lemma 4.7, we have
$$
|i^d-j^d|\leq c_{10}|k|,~~~|k|\ge 16M.
$$
For any fixed $i\in {\bf N}_+$, we have
\begin{eqnarray}
&~&{\rm card}\{j:|i^d-j^d|\le c_{10}|k|\}\nonumber\\&\le& {\rm
card}\{j:~j\le
i\}+{\rm card}\{j:~j^d\le c_{10}|k|+i^d\}\nonumber\\
&\le& i+{\rm card}\{j:~j^d\le (c_{10}+1)|k|i^d\}\le
i+c|k|^{\frac1d}i\nonumber\\
&\le&(c_{10}+1)|k|^{\frac1d}i.\nonumber
\end{eqnarray}
Then for fixed $i$,
$$
|R^{\nu+1}_{kl(i)}|\le |R^{\nu+1}_{kl}|{\rm card\{j:|i^d-j^d|\le
c_8|k|\}}\le c(n)\frac{\gamma_0 |k|^{\frac
1d}}{(1+|k|^{\tau})i^{\frac32}},
$$
and
\begin{eqnarray}
|I|&\leq&\sum_{i\ge1}^{\infty}|R_{kl(i)}^{\nu+1}|\prec
c(n)\frac{\gamma_{0}}{(1+|k|^{\tau})|k|^{1-\frac1d}}\sum_{i\ge1}\frac{1}{i^{\frac32}}\nonumber\\
&\prec& c(n)\frac{\gamma_{0}}{(1+|k|^{\tau})}.\nonumber
\end{eqnarray}
The case $l\in\Lambda_+$ is the same, we omit the details.

Then for the case $\rho=1$, we set $c(\rho)=\frac52$, and the
measure estimate is the following:
\begin{eqnarray}
|\Pi\setminus\Pi_{*}|&\leq&\sum_{0}^{\nu_{*}}\tilde{R}_{\nu+1}
+\sum_{\nu_{*}}^{\infty}\sum_{_{K_{\nu}<|k|<K_{\nu+1}}}\tilde{R}_{\nu+1}\nonumber\\
&\le& O(\gamma_{0}),\nonumber
\end{eqnarray}
i.e. $|\Pi\setminus\Pi_{*}|\rightarrow 0$, as $\gamma_0\rightarrow
0$. \qquad\end{proof}

\section{Applications to PDEs}

In this section, we shall give two applications of our result: 1. The nonlinear Schr\"odinger equations with Dirichlet boundary condition.
2. The Klein-Gordon equations with exponential nonlinearity subject to Dirichlet boundary condition. Due to the work of \cite{peter}, we can see that both equations possess weaker spectral asymptotics in higher dimension. We remark that those results can be extended to the nonlinear Schr\"odinger equations (or the Klein-Gordon equations with exponential nonlinearity) with other boundary conditions.

\subsection{Nonlinear Schr\"odinger Equations}

Consider the Schr\"odinger equations
\begin{equation}
\sqrt{-1}u_{t}-\bigtriangleup u+mu=\frac{\partial F}{\partial
\bar{u}},~~x\in\Omega,~~t\in{\bf R}\label{5.1}
\end{equation}
with the Dirichlet boundary condition
\begin{equation}
u|_{\partial\Omega}=0,\label{5.2}
\end{equation}
where $\Omega\subset{\bf R}^{m}$ is a bounded domain with smooth
boundary $\partial\Omega$, $m\neq0$ is a constant,
$F$ is real analytic and $\frac{\partial F}{\partial
\bar{u}}=f(|u|^{2})u$ with $ f(0)=0,~ f'(0)\ne0$, $\xi$ is defined
as before.

Consider the  eigenvalues problem as following:
$$-\bigtriangleup\phi=\lambda\phi,$$
Weyl's asymptotic formula asserts
 $$\lambda^{*}_{j}\sim
C_{m}(\frac{j}{|\Omega|})^{2/m},~~~k\rightarrow~\infty,$$ where
$|\Omega|$ is the volume of $\Omega$, $C_{m}=(2\phi)^2 B_{m}^{-2/m}$
is the Weyl constant. See \cite{peter}.

For fixed $\xi\in\Pi$, the eigenvalues of the operator
$-\triangle+V(\xi)$ has the following asymptotic behavior, as
$j\rightarrow\infty$,
\begin{eqnarray}\label{tezhenggen}
\lambda_{j}\sim c j^{2/m}+o(j^{2/m})~~~~j\rightarrow~\infty,
\end{eqnarray}
where $c=\frac{C_{m}}{|\Omega|^{2/m}}$.

Denote that, the eigenvalues of  operator $A=-\triangle+m$
under the boundary condition (\ref{5.2}) satisfy
\begin{eqnarray}
\bar{\omega}_{j}&=&\lambda_{j}(\xi),~~~~~j=\{1,\cdots,n\}\nonumber\\
\bar{\Omega}^{j}&=&\lambda_{j}(\xi),~~~~~j\in N\equiv {\bf
N}\setminus\{1,\cdots,n\}.\nonumber
\end{eqnarray}

Rewrite (\ref{5.1}) as
$$
u_{t}=\sqrt{-1}\frac{\partial H}{\partial \bar{u}},
$$
with the associated  Hamiltonian \begin{equation}\label{5.3}
H=\langle Au,u\rangle+\int_{\Omega}F(u){\rm d}x,
\end{equation}
where $\langle \cdot,\cdot\rangle$ denote the inner product in
$L^{2}$.

Consider
$$
u(t,x)=\sum_{j\ge1}q_{j}(t)\phi_j(x),
$$
where $\phi_j(x)$ is the eigenvector corresponding to eigenvalue
$\lambda_j$. Then (\ref{5.1}) can be transformed as
\begin{equation}
\dot{q_{j}}=\sqrt{-1}(\lambda_jq_j+\frac{\partial
G}{\partial\bar{q}_{j}}),~~~j\ge1,
\end{equation}
where $G=\int_{\Omega}F(u){\rm d}x$, and (\ref{5.3}) becomes
\begin{equation}\label{H(q,barq)}
H=\sum_{j\ge1}\lambda_jq_j\bar{q}_j+G(q,\bar{q},\phi,\bar{\phi}).
\end{equation}

To apply Theorem A, we need to consider the regularity of the
nonlinearity $G$. Introduce the  Hilbert space
$l^{2}=(\cdots,q_{-j},\cdots,q_{j},\cdots)$ denotes all bi-infinite
square summable sequences with complex coefficients. $L^{2}$ denotes
all square-integrable complex-valued functions on ${\bf T}^{n}$, by
the inverse Fourier transform
$$
{\mathcal {F}}:~~q\rightarrow{\mathcal
{F}}q=\frac{1}{\sqrt[2n]{2\pi}}\sum_{j}q_j\phi_j(x).
$$

Let $a\ge0,~p\ge0$. The subspace $l^{a,p}\subset l$ consist of all
bi-infinite sequence with finite form
$$
(|q|^{a,p})^{2}=|q_0|^{2}+\sum_{j}|q_j||j|^{2p} {\rm
e}^{2a|j|}<\infty.
$$
Denote function space $W^{a,p}\subset L^{2}$ normed by $|{\mathcal
{F}}q|^{a,p}=|q|^{a,p}$. For $a>0$, the space $W^{a,p}$ may be
analytic functions bounded in the complex strip $|{\rm Im}x|<a$ with
trace functions on $|{\rm Im}x|=a$ belonging to the usual Sobolev
space $W^{p}$. (see \cite{kuksin1})

\begin{lemma}~ For any fixed $a\ge0,~p\ge0$, the gradient $G_{\bar{q}}$ is
real analytic as a map in a neighborhood of the origin with
\begin{eqnarray*}
|G_{\bar{q}}|^{a,p}\le c (|q|^{a,p})^{3}.
\end{eqnarray*}
\end{lemma}
\begin{proof}~Consider a function $u\in W^{a,p}$ with the norm
$|u|^{a,p}=|q|^{a,p}$. The function $f(|u|^{2})u$ also belongs to
$W^{a,p}$ with
$$
|f(|u|^{2})u|^{a,p}\le c(|u|^{a,p})^{3},
$$
in a sufficiently small neighbourhood of the origin, for details,
see \cite{kuksin1}.\qquad\end{proof}

Consider the subspace $l^{2\bar{a}-a,2\bar{p}-p}\subset
l^{\bar{a},\bar{p}}\subset l^{a,p}$, and the corresponding function
subspaces $W^{2\bar{a}-a,2\bar{p}-p}\subset
W^{\bar{a},\bar{p}}\subset W^{a,p}\subset L^{2}$, where
$0<a<\bar{a},~0\le p\le\bar{p}$.

Let $u=\sum_{j}q_{j}\phi_j(x)$. If $u\in
 W^{\bar{a},\bar{p}}\subset W^{a,\bar{p}}$, then
according to Lemma 5.3, we have
\begin{eqnarray}\label{Gbarp}
|G_{\bar{q}}|^{\bar{a},\bar{p}}\le c (|q|^{\bar{a},\bar{p}})^{3},
\end{eqnarray}
and
$$
|q|^{2\bar{a}-a,2\bar{p}-p}\le C<\infty.
$$
By the Schwarz inequality,
\begin{eqnarray*}
(|q|^{\bar{a},\bar{p}})^{2}&=& \sum_{j}|q_j|^{2}[j]^{2\bar{p}}{\rm
e}^{2\bar{a}|j|}=\sum_{j}(|q_j|[j]^{p}{\rm
e}^{a|j|})(|q_j|[j]^{2\bar{p}-p}{\rm e}^{(\bar{a}-a)|j|})\\
&\le& \sqrt{\sum_{j}|q_j|^{2}[j]^{2p}{\rm
e}^{2a|j|}}\sqrt{\sum_{j}|q_j|^{2}[j]^{2\bar{p}-p}{\rm
e}^{(2\bar{a}-a)|j|}} \le C|q|^{a,p}.
\end{eqnarray*}
Combining with (\ref{Gbarp}) yields
$$
|G_{\bar{q}}|^{\bar{a},\bar{p}}\le c (|q|^{a,p})^{\frac32},
$$
where $c$ depends on $a,~\bar{a},~p,~\bar{p}$. The regularity of
$X_{G}$ follows from the regularity of $G_{\bar{q}},~G_{q}$.
Similarly, its Lipschtiz semi-norm with respect to $\xi$ follows.
Choose $s$ such that the perturbation satisfies A3).

In the following, we introduce the usual action-angle variables
$(\varphi,I)\in{\bf T}^{n}\times{\bf R}^{n}$ and
infinite-dimensional vector $z,~\bar{z}\in{\mathcal
{L}}^{a,p}\times{\mathcal {L}}^{a,p}$ (see \cite{geng}). Then,
systems (\ref{system1}) becomes
$$
\left\{\begin{array}{lllllllll}
\dot \varphi_{j} & =  \bar{\omega}_{j}(\xi)+G_{I_{j}},~~&j&\in\{1,\cdots,n\},\\
\dot I_{j} & = -G_{\varphi_{j}},~~&j&\in\{1,\cdots,n\},\\
\dot z_{j} & =(\bar{\Omega}_{j}(\xi)\bar{z}_{j}+G_{\bar{z}_{j}}),~~&j&\in N,\\
\dot{\bar{z}}_{j}&=-(\bar{\Omega}_{j}(\xi)z_{j}+G_{z_{j}}),~~&j&\in
N,
\end{array}\right.
$$
associated to the Hamiltonian
$$
H=\langle\bar{\omega}(\xi),I\rangle+\sum_{j\in
N}\bar{\Omega}_j(\xi)z_{j}\bar{z}_j+G(I,\varphi,z,\bar{z}),
$$
where $(x,y,z,\bar{z})$ lie in complex neighborhood
$$D_{a,p}(s,r)=\{(x,y,z,\bar{z}):|{\rm Im}~x|<r, |y|<s^2,
|z|^{a,p},|\bar{z}|^{a,p}<s\}$$of ${\bf T}^{n}\times \{0\}\times
\{0\}\times \{0\}\subset{\bf T}^{n}\times {\bf R}^{n}\times {\bf
\mathcal {L}}^{a,p}\times {\bf \mathcal {L}}^{a,p}$, and with
respect to the symplectic form
$$
\sum_{i=1}^{n}{\rm d}x_{i}\wedge {\rm d}y_{i}+\sum_{j\in N}{\rm
d}z_{j}\wedge {\rm d}\bar{z}_{j}.
$$

Note that the coefficient of $j^{2/m}$ can always be normalized to
one by rescaling. It is easily to prove that the frequences
$\bar{\omega}(\xi),~\bar{\Omega}_{j}(\xi)$ satisfy A1), A2)  after
rescaling. Using Theorem A,  we  get the following result:

\vspace{0.5cm} \noindent{\bf Theorem 5.1} {\it For any
$0<\gamma\le1$, there exists a Cantor set $\Pi_{\gamma}\subset\Pi$,
with $|\Pi\setminus\Pi_{\gamma}|= O(\gamma)$, such that for any
$\xi\in\Pi_{\gamma}$, Schr\"odinger equations (\ref{5.1}) subjected
to the boundary condition (\ref{5.2})
 admits a family of small amplitude quasi-periodic solutions
 $u_{*}(t,x)$ with respect to time $t$. Moreover
 $u_{*}(t,x)\in W_{b}^{a,\bar{p}}$ for fixed $t$.
 }

\vspace{0.5cm}
\subsection{the Klein-Gordon equations with exponential nonlinearity}
In this subsection, using Theorem A, we show that the existence of quasi-periodic solutions for the Klein-Gordon equations with exponential nonlinearity subject to Dirichlet boundary conditions
\begin{equation*}
u_{tt}-\bigtriangleup u+M_{\xi}u=ue^{\alpha u^2},~~x\in\Omega,~~t\in{\bf R},\label{wave}
\end{equation*}
\begin{equation}
u|_{\partial\Omega}=0,\label{6.2}
\end{equation}
where $M_{\xi}$ is a real Fourier multiplier,
\begin{eqnarray*}
M_{\xi}\sin jx=\xi_j\sin jx,~~\xi_j\in R^{\rho}.
\end{eqnarray*}
Exponential-type nonlinearities have been considered in several physical models (see, e.g., \cite{Lam} on a model of self-trapped beams in plasma).
S. Ibrahim, M. Majdoub and N. Masmoudi \cite{Ib1} has obtained
the existence of global solutions for the Klein-Gordon equations with exponential nonlinearity in two dimension. Recently, S. Ibrahim, M. Majdoub, N. Masmoudi and K. Nakanishi \cite{Ib2} investigated the existence and asymptotic completeness of the wave operators for the
nonlinear Klein-Gordon equation with a defocusing exponential nonlinearity in two
space dimensions. In this paper, we will prove the existence of quasi-periodic solutions for the Klein-Gordon equations with exponential nonlinearity.

Let
\begin{eqnarray*}
\frac{\partial F}{\partial u}:=u(e^{\alpha u^2}-1).
\end{eqnarray*}
Then we can rewrite the Klein-Gordon equations with exponential nonlinearity as
\begin{equation}
u_{tt}-\bigtriangleup u+(M_{\xi}-1)u=\frac{\partial F}{\partial u},\label{wave}
\end{equation}
In order to obtain the existence of quasi-periodic solutions,  we need assume:
\begin{itemize}
\item[{\bf I)}] Then operator $A=-\bigtriangleup+M_{\xi}-1$ under the
boundary condition (\ref{6.2}) admits the spectrum
\begin{eqnarray}
\omega_{\jmath}&=&\lambda_{{\imath}_{\jmath}}=|{\imath}_{\jmath}|+\xi_j-1,~~~~~~~~~1\le\jmath\le n,\nonumber\\
\Omega_{j}&=&\lambda_j=|j|+o(|j|^{-1}),~~j\in
Z^{\rho}\setminus\{\imath_1,\cdots,\imath_n\}.\nonumber
\end{eqnarray}
The corresponding eigenfunctions
$\phi_j=\frac{1}{\sqrt{(2\pi)^{\rho}}}e^{\sqrt{-1}}\langle
j,x\rangle$ form a basis. Without loss of generality, we assume
$0\in\{\imath_1,\cdots,\imath_n\}$.
\end{itemize}

Introduce $v=u_t$, and (\ref{wave}) reads
$$
\left\{\begin{array}{lllllllll}
 u_{t}&=&v,\\
 v_{t}&=&-Au-\frac{\partial
F}{\partial u},
\end{array}\right.
$$
with the associated to the Hamiltonian
$$
H=\frac12(\langle v,v\rangle+\langle Au,u\rangle)+\int_{{\bf
T}^{\rho}}F(u)dx.
$$
Consider the coordinate transformation:
\begin{eqnarray*}
u&=&\sum_{j\in
Z^{\rho}}\frac{q_{j}}{\sqrt{\lambda_{j}}}\phi_j(x),~v=\sum_{j\in
Z^{\rho}}\sqrt{\lambda_{j}}p_{j}\phi_j(x),
\end{eqnarray*}
where $\phi_j(x)$ is the eigenvector corresponding to eigenvalue
$\lambda_j$. Then (\ref{wave}) can be transformed to
\begin{equation}\label{system1}
\dot{p_{j}}=-\frac{\partial H}{\partial
q_{j}},~~\dot{q_{j}}=\frac{\partial H}{\partial p_{j}},~~j\in
Z^{\rho},
\end{equation}
and  the corresponding Hamiltonian is
$$
H=\frac12\sum_{j}\lambda_j(p^{2}_{j}+q^{2}_{j})+\int_{{\bf
T}^{\rho}}F(\sum_{j}\frac{q_n}{\sqrt{\lambda_j}}){\rm d}x.
$$

Introduce complex variables
$w_j=\frac{1}{\sqrt{2}}(q_j+\sqrt{-1}p_j),~\bar{w}_n=\frac{1}{\sqrt{2}}(q_j-\sqrt{-1}p_j)$.
Then the perturbation of equation (\ref{system1}) reads
\begin{equation}\label{GWW}
G(w,\bar{w})\equiv\int_{{\bf T}^{\rho}}F(\sum_{j\in
Z^{\rho}}\frac{w_j\phi_j+\bar{w}_j\bar{\phi}_j}{\sqrt{2\lambda_j}}){\rm
d}x.
\end{equation}
Next using usual action-angle variables $(\varphi,I)\in{\bf
T}^{\rho}\times{\bf R}^{\rho}$ and infinite-dimensional vector
$w_j=z_j,~\bar{w}_j=\bar{z}_j,~j\ne\imath_1,\cdots,\imath_n$. Then,
systems (\ref{system1}) becomes
$$
\left\{\begin{array}{llllllll} \dot \varphi_{\jmath} &=&
\omega_{\jmath}(\xi)+G_{I_{j}},
&~&\dot I_{\jmath}=-G_{\varphi_{\jmath}},~~&\jmath&\in\{\imath_1,\cdots,\imath_n\}\\
\dot z_{j}&=&(\Omega_{j}(\xi)\bar{z}_{j}+G_{\bar{z}_{j}}),~~
&~&\dot{\bar{z}}_{j}=-(\Omega_{j}(\xi)z_{j}+G_{z_{j}}),~~&j&\in
Z^{\rho}_1,
\end{array}\right.
$$
associated to the Hamiltonian
\begin{equation}\label{hamilton1}
H=\langle\omega(\xi),I\rangle+\sum_{j\in
Z_1^{\rho}}\Omega_j(\xi)z_{j}\bar{z}_j+G(I,\varphi,z,\bar{z}),
\end{equation}
where $(x,y,z,\bar{z})$ lies in complex neighborhood
$$D_{a,p}(s,r)=\{(x,y,z,\bar{z}):|{\rm Im}~x|<r, |y|<s^2,
|z|^{a,p},|\bar{z}|^{a,p}<s\}$$of ${\bf T}^{n}\times \{0\}\times
\{0\}\times \{0\}\subset{\bf T}^{n}\times {\bf R}^{n}\times {\bf
\mathcal {L}}^{a,p}\times {\bf \mathcal {L}}^{a,p}$, and with
respect to the symplectic form
$$
\sum_{i=1}^{n}{\rm d}x_{i}\wedge {\rm d}y_{i}+\sum_{j\in
Z_1^{\rho}}{\rm d}z_{j}\wedge {\rm d}\bar{z}_{j}.
$$

To apply our main theorem, we need to verify the perturbation $P$
satisfies A3)' and A4)'. (see section 1)

 Introduce the  Hilbert space $l^{2}=(\cdots,l_n,\cdots)_{n\in Z^{\rho}}$
and $L^{2}$ denotes all square-integrable complex-valued functions
on ${\bf T}^{\rho}$, by the inverse Fourier transform
$$
{\mathcal {F}}:~~q\rightarrow{\mathcal
{F}}q=\frac{1}{\sqrt{{2\pi}^{\rho}}}\sum_{j}q_j\phi_j(x).
$$

Let $a\ge0,~p\ge0$. The subspace $l^{a,p}\subset l^2$ consist of
sequences with finite form
$$
(|q|^{a,p})^{2}=|q_0|^{2}+\sum_{j}|q_j||j|^{2p} {\rm
e}^{2|j|a}<\infty.
$$
Denote function space $W^{a,p}\subset L^{2}$ normed by $|{\mathcal
{F}}q|^{a,p}=|q|^{a,p}$. For $a>0$, the space $W^{a,p}$ may be
analytic functions bounded in the complex strip $|{\rm Im}x|<a$ with
trace functions on $|{\rm Im}x|=a$ belonging to the usual Sobolev
space $W^{p}$. (see \cite{kuksin1})

Consider the subspace $l^{2\bar{a}-a,p}\subset l^{\bar{a},p}\subset
l^{a,p}$, and $u\in W^{2\bar{a}-a,p+\varsigma}\subset
W^{\bar{a},p+\varsigma}\subset W^{a,p+\varsigma}$ respectively with
the norm above, where $\bar{a}>a>0,~p\ge0$.

\begin{lemma}~ For any fixed $a\ge0,~p\ge0$, the gradient $G_{\bar{q}}$ is
real analytic as a map in a neighborhood of the origin with
\begin{eqnarray}
|G_{\bar{q}}|^{a,p+\varsigma}\le c (|q|^{a,p})^{3},
\end{eqnarray}
where $\varsigma\le\frac{1}{2}$.
\end{lemma}
\begin{proof}~Let $q\in l^{a,p}$. Then
$u=\sum_{j}\frac{q_{j}}{\lambda_{j}}\phi_j$ is in
$W^{a,p+\varsigma}$ with $|u|^{a,p+\varsigma}=|q|^{a,p}$. One can
see \cite{posel2} for details.\qquad\end{proof}

\begin{lemma}
If $u\in W^{2\bar{a}-a,p+\varsigma}\subset
W^{\bar{a},p+\varsigma}$, we have
\begin{equation}\label{Gbarp2}
|G_{\bar{q}}|^{\bar{a},p+\varsigma}\le c (|q|^{a,p})^{\frac32},
\end{equation}
where $c$ depends on $a,~\bar{a},~p,~\bar{p}$.
\end{lemma}
\begin{proof} According to Lemma 5.1 and the Schwarz inequality,
\begin{eqnarray*}
(|q|^{\bar{a},p})^{2}&=& \sum_{j}|q_j|^{2}[j]^{2p}{\rm
e}^{2\bar{a}|j|}=\sum_{j}(|q_j|[j]^{p}{\rm
e}^{a|j|})(|q_j|[j]^{p}{\rm e}^{(\bar{a}-a)|j|})\\
&\le& \sqrt{\sum_{j}|q_j|^{2}[j]^{2p}{\rm
e}^{2a|j|}}\sqrt{\sum_{j}|q_j|^{2}[j]^{2p}{\rm
e}^{(2\bar{a}-a)|j|}}.
\end{eqnarray*}

Since $q\in l^{2\bar{a}-a,p}$, we have
$$
\sqrt{\sum_{j}|q_j|^{2}[j]^{2p}{\rm e}^{(2\bar{a}-a)|j|}}\le
c(a,\bar{a},p,\bar{p}),
$$
then (\ref{Gbarp2}) is verified.\qquad\end{proof}

\noindent {\bf Remark} The regularity of perturbation $X_{G}$
depends not only on the order of $u$ but also on the weight of the
function space $W^{a,p}$. In previous, one assumed $u\in
W^{a,\bar{p}}$ and proved that the quasi-periodic solution $u_{*}$
of the perturbed system has the finite $a,\bar{p}$-norm for fixed
$t$. In our paper, we demand $u\in W^{2\bar{a}-a,\bar{p}}$, and
prove that the perturbed quasi-periodic solution $u_{*}$ is in
$W^{a,\bar{p}}$. The more rigorous regularity of the perturbation is
to overcome the continuous spectrum. This is just intrinsic.

The regularity of $X_{G}$ follows from the regularity of
$G_{\bar{q}},~G_{q}$. Similarly, its Lipschtiz semi-norm follows.
Choosing the parameter $s$,  A3) is verified.

Then we show that $P$ has the special form.
Since $e^{\alpha u^2}$ is real analytic in $u$, making use of $q(x)=\sum_{j\in
Z^{\rho}}q_n\phi_n(x)$, $F$ can be rewritten as
$$
F(w,\bar{w})=\sum_{\alpha,\beta}F_{\alpha,\beta}w^{\alpha}\bar{w}^{\beta}\phi^{\alpha}\bar{\phi}^{\beta},
$$
hence
\begin{eqnarray}
G(w,\bar{w})&\equiv&\int_{{\bf T}^{\rho}}F(\sum_{j\in
Z^{\rho}}\frac{w_j\phi_j+\bar{w}_j\bar{\phi}_j}{\sqrt{2\lambda_j}}){\rm
d}x\nonumber\\
&=&\sum_{\alpha,\beta}G_{\alpha\beta}\int_{{\bf T}^{\rho}}{\rm
e}^{\sqrt{-1}\langle(\sum_{j\in
Z^{\rho}}\alpha_j-\beta_j)j,x\rangle}{\rm d}x,\nonumber
\end{eqnarray}
that is
$$
G_{\alpha\beta}=0,~~{\rm if}~\sum_{j}(\alpha_j-\beta_j)j\ne0.
$$

Denote
$k=(k_1,\cdots,k_n),~k_{\jmath}=\alpha_{\jmath}-\beta_{\jmath},~1\le\jmath\le
n$. Then
\begin{eqnarray}
&~&G=\sum_{\sum_{j\in
Z^{\rho}}\sum_{j}(\alpha_j-\beta_j)j\ne0}G_{\alpha\beta}w^{\alpha}\bar{w}^{\beta}\label{GOO}\\
&=&\sum_{\sum_{{\jmath}=1}^{n}(\alpha_{\jmath}-\beta_{\jmath}){\imath}_{\jmath}+\sum_{j\in
Z_{1}^{\rho}}(\alpha_j-\beta_j)j\ne0}G_{\alpha\beta}q^{\alpha_{{\imath}_1}}_{{\imath}_1}
\bar{q}^{\beta_{{\imath}_1}}_{{\imath}_1}\cdots
q^{\alpha-\sum_{\jmath=1}^{n}\alpha_{{\imath}_{\jmath}}e_{\imath_{\jmath}}}
\bar{q}^{\beta-\sum_{\jmath=1}^{n}\beta_{{\imath}_{\jmath}}e_{\imath_{\jmath}}}\nonumber\\
&\equiv&\sum_{\sum_{{\jmath}=1}^{n}k_{\jmath}\imath_{\jmath}+\sum_{j\in
Z_{1}^{\rho}}(\alpha_j-\beta_j)j\ne0}P_{km\alpha\beta}z^{\alpha}\bar{z}^{\beta}\equiv
P,\nonumber
\end{eqnarray}
that is, $P\in {\mathcal {A}}$. Therefore, Hamiltonian PDEs in $\Omega$ not containing explicitly the space and time variable do
have the special form.

Thus we have verified A1)'-A4)'. Then we have the following result
for the Klein-Gordon equations with exponential nonlinearity.
 \vspace{0.5cm}

 \noindent{\bf Theorem 5.2} {\it Under the assumptions
{\rm I)}, for any $0<\gamma\le1$, there exists a Cantor set
$\Pi_{\gamma}\subset\Pi$, with $|\Pi\setminus\Pi_{\gamma}|=
O(\gamma)$, such that for any $\xi\in\Pi_{\gamma}$, wave equations
(\ref{wave}) subjected to the boundary condition (\ref{6.2})
 admits a family of small amplitude quasi-periodic solutions
 $u_{*}(t,x)$ with respect to time $t$. Moreover
 $u_{*}(t,x)\in W_{b}^{a,\bar{p}}$ for fixed $t$.
 }


\begin{thebibliography}{xx}

\bibitem{berti1}
{\sc Berti. Massimiliano, Bolle. Philippe}, Sobolev periodic
solutions of nonlinear wave equations in higher spatial dimensions,
 Arch. Ration. Mech. Anal. \textbf{195} 2010, 609-642

\bibitem{berti2}
{\sc Berti. Massimiliano, Bolle. Philippe}, Cantor families of
periodic solutions for wave equations via a variational principle,
Adv. Math. \textbf{217} 2008, 1671-1727

\bibitem{berti3}
{\sc Berti. Massimiliano, Procesi. Michela}, Quasi-periodic
solutions of completely resonant forced wave equations, Comm.
Partial Differential Equations. \textbf{31} 2006, 959-985


\bibitem{bourgain1}
{\sc J. Bourgain}, Construction of quasi-periodic solutions for
Hamiltonian perturbations of linear equations and application to
nonlinear PDE, Internat. Math. Res. Notices. \textbf{11} 1994, 475-497

\bibitem{bourgain2}
{\sc J. Bourgain}, On Melnikov’s persistency problem,  Math.
Res. Lett. \textbf{4} 1997, 445-458

\bibitem{bourgain3}
{\sc J. Bourgain}, Quasi-periodic solutions of Hamiltonian
perturbations for 2D linear Schr\"odinger equation,  Ann. Math. \textbf{148} 1998,
363-439

\bibitem{B}
{\sc J. Bourgain}, Green's Function Estimates for Lattice Schr\"{o}inger Operators and Applications,
Ann. of Math. Stud., vol. \textbf{158}, Princeton University Press, Princeton, NJ, 2005.

\bibitem{B1}
{\sc J. Bourgain}, Construction of periodic solutions of nonlinear wave equations in higher dimension, Geom. Funct.
Anal. \textbf{5} 1995, 629-639

\bibitem{chie}
{\sc L. Chierchia and Jiangong You}, KAM Tori for 1D Nonlinear Wave
Equations with Periodic Boundary Conditions, Commun. Math. Phys. \textbf{211}
2000 497-525

\bibitem{craig}
{\sc W. Craig, C. Wayne}, Newton’s method and periodic solutions of
nonlinear wave equation, Comm. Pure. Appl. Math. \textbf{46} 1993, 1409-1501

\bibitem{eliasson}
{\sc L.H. Eliasson, S. B. Kuksin}, KAM for the nonlinear
Schr\"odinger equation, Anna. of Math. \textbf{172} 2010, 371-435

\bibitem{geng}
{\sc J.S. Geng, J.G. You} A KAM theorem for Hamiltonian partial differential equations in higher dimensional space, Commun. Math. Phys. \textbf{262}
2006 343-372

\bibitem{geng1}
{\sc J.S. Geng, X.D. Xu and J.G. you}, An infinite dimensional KAM theorem and its
application to the two dimensional cubic Schr\"{o}dinger
equation, Advances in Mathematics. \textbf{226} 2011, 5361-5402

\bibitem{Ib1}
{\sc S. Ibrahim, M. Majdoub, N. Masmoudi}, Global solutions for a semilinear, two dimensional Klein-Gordon equation with exponential-type nonlinearity,
Comm. Pure. Appl. Math. \textbf{59} 2006, 1639-1658

\bibitem{Ib2}
{\sc S. Ibrahim, M. Majdoub, N. Masmoudi, K. Nakanishi},
Scattering for the two-dimensional energy-critical wave equation,
Duck. Math. J. \textbf{150} 2009, 287-329

 \bibitem{kuksin}
{\sc S.B. Kuksin}, Nearly integrable infinite dimensional
Hamiltonian systems, Lecture Notes in Mathematics, Springer, Berlin,
1993.

\bibitem{kuksin1}
{\sc S.B. Kuksin, J.P\"oschel}, Invariant Cantor manifolds of
quasi-periodic oscillations for a nonlinear Schr\"{o}dinger equation,
 Ann. of Math. \textbf{142} 1995, 149-179

\bibitem{Lam}
{\sc J. F. Lam, B. Lippmann, F. Tappert}, Self-trapped laser beams in plasma, Phys.
Fluids \textbf{20} 1977, 1176-1179

\bibitem{peter}
{\sc Peter Li, Shing-Tung Yuan} On the Schr\"{o}ding Equation and the
Eigenvalue Problem, Commun. Math. Phys. \textbf{88} 1983, 309-318

\bibitem{posel1}
{\sc J. P\"{o}schel}, A KAM-Theorem for some Nonlinear Partial
Differential Equation.  Ann. Sucola Norm. Sup. Pisa Cl. Sci. \textbf{23} 1996,
119-148

\bibitem{posel2}
{\sc J. P\"{o}schel}, Quasi-periodic sulutions for a nonlinrar wave
equation.  Comment. Math. Helv. \textbf{71} 1996, 269-296

\bibitem{wayne}
{\sc C.E. Wayne}, Periodic and quasi-periodic solutions for
nonlinear wave equations via KAM theory, Commun. Math. Phys. \textbf{127} 1990,
479-528

\end{thebibliography}
\end{document}